\documentclass[
11pt,                          
final,                         
english                        
]{article}

\usepackage{amsmath}           
\usepackage{amssymb}           
\usepackage[utf8]{inputenc}    
\usepackage[T1]{fontenc}       
\usepackage[a4paper]{geometry} 
\usepackage{setspace}          
\usepackage{bbm}               
\usepackage{mathrsfs}          
\usepackage{mathtools}         
\usepackage{stmaryrd}          
\usepackage[thmmarks,hyperref]{ntheorem} 
\usepackage{xspace}            
\usepackage[normalem]{ulem}    
\usepackage[activate={true,nocompatibility},
            final,tracking=true,kerning=true,
            spacing=true,factor=1100,
            stretch=10,shrink=10,expansion=false
           ]{microtype}
\microtypecontext{spacing=nonfrench}  
\usepackage[hyperref]{xcolor}
\definecolor{mydarkblue}{RGB}{0,0,155}
\usepackage[backref=page,      
           final=true,         
           colorlinks=true,    
           allcolors=mydarkblue,  
           hypertexnames=false,
           plainpages=false,              
           pdfpagelabels=true,            
           pdfencoding=auto,              
           unicode=true                   
                      ]{hyperref}         
\usepackage[shortcuts]{extdash}  

\author{
  \textbf{Matthias Schötz}\thanks{Former Boursier de l'ULB.
  This work was supported by the Fonds de la Recherche Scientifique (FNRS) and the Fonds Wetenschappelijk
  Onderzoek - Vlaaderen (FWO) under EOS Project n$^0$30950721.
  Current Adresse: Universität Leipzig, \href{mailto:Matthias.Schoetz@math.uni-leipzig.de}{\texttt{Matthias.Schoetz@math.uni-leipzig.de}}.}\\
  Département de mathématiques\\
  Université libre de Bruxelles
}

\newcommand{\refitem}[1] {\textit{\ref{#1}.)}}

\numberwithin{equation}{section}

\theoremheaderfont{\normalfont\bfseries}
\theorembodyfont{\itshape}
\newtheorem{lemma}{Lemma}[section]

\newtheorem{proposition}[lemma]{Proposition}
\newtheorem{theorem}[lemma]{Theorem}

\newtheorem{corollary}[lemma]{Corollary}
\newtheorem{definition}[lemma]{Definition}
\theorembodyfont{\rmfamily}
\newtheorem{example}[lemma]{Example}
\theoremheaderfont{\normalfont\bfseries}
\theorembodyfont{\normalfont}
\theoremstyle{nonumberplain}
\theoremseparator{:}
\theoremsymbol{\hbox{$\boxempty$}}
\newtheorem{proof}{Proof}

\newcommand{\Unit}           {\mathbbm{1}}
\newcommand{\argument}       {\ignorespaces{\,\cdot\,}\ignorespaces}
\DeclarePairedDelimiter{\abs}{\lvert}{\rvert}
\DeclarePairedDelimiter{\norm}{\lVert}{\rVert}
\newcommand{\cc}[1]          {\overline{{#1}}}
\newcommand{\I}              {\mathrm{i}}
\newcommand{\NN}{\mathbbm{N}}
\newcommand{\RR}{\mathbbm{R}}
\newcommand{\CC}{\mathbbm{C}}
\newcommand{\FF}{\mathbbm{F}}
\DeclarePairedDelimiter{\ordinaryIP}{\langle}{\rangle}
\newcommand{\skal}[3][]{\ordinaryIP[#1]{\,#2 \,#1|\, #3\,}}
\newcommand{\dupr}[3][]{\ordinaryIP[#1]{\,#2 \,,\, #3\,}}
\DeclarePairedDelimiter{\ordinarySet}{\{}{\}}
\newcommand{\set}[3][]{\ordinarySet[#1]{\,#2 \;#1|\; #3\,}}
\newcommand{\seminorm}[3][]{\norm[#1]{#3}_{#2}}
\newcommand{\Lin}{\mathcal{L}}
\newcommand{\Adbar}{\mathcal{L}^*}
\newcommand{\States}{\mathcal{S}}
\newcommand{\PureStates}{\mathcal{S}_{\mathrm{p}}}
\newcommand{\MultStates}{\mathcal{S}_{\mathrm{m}}}
\DeclareFontFamily{U}{FdSymbolF}{}
\DeclareFontShape{U}{FdSymbolF}{m}{n}{
    <-7.1> FdSymbolF-Book
    <7.1-> FdSymbolF-Book
}{}
\DeclareSymbolFont{delimiters}{U}{FdSymbolF}{m}{n}
\DeclareMathDelimiter{\llangle}{\mathopen}{delimiters}{"92}{delimiters}{"92}
\DeclareMathDelimiter{\rrangle}{\mathclose}{delimiters}{"98}{delimiters}{"98}
\DeclarePairedDelimiter{\coolIP}{\llangle}{\rrangle}
\newcommand{\clconvHull}[2][]{\coolIP[#1]{#2}_{\mathrm{cl\textup{-}conv}}}
\newcommand{\filterHull}[2][]{\coolIP[#1]{#2}_{\mathrm{fltr}}}
\newcommand{\ex}{\mathrm{ex}}
\newcommand{\Hermitian}{\mathrm{H}}
\newcommand{\Dom}{\mathcal{D}}
\newcommand{\Var}{\mathrm{Var}}

\title{%
Gelfand--Naimark Theorems for Ordered \texorpdfstring{$^*$-Algebras}{*-Algebras}%
}

\date{March 2022}

\begin{document}
\begin{onehalfspace}

\maketitle

\begin{abstract}
  The Gelfand--Naimark theorems provide important insight into the structure of
  general and of commutative $C^*$\=/algebras. It is shown that these can be generalized
  to certain ordered $^*$\=/algebras.
  More precisely, for $\sigma$\=/bounded closed ordered $^*$\=/algebras a faithful representation as operators is constructed.
  Similarly, for commutative such algebras, a faithful representation as complex-valued functions is constructed if an additional necessary regularity condition is fulfilled.
  These results generalize the Gelfand--Naimark representation theorems to classes of $^*$\=/algebras larger than $C^*$\=/algebras, and which especially
  contain $^*$\=/algebras of unbounded operators.
  The key to these representation theorems is a new result for Archimedean ordered vector spaces $V$:
  If $V$ is $\sigma$\=/bounded, then the order of $V$ is induced by the extremal positive linear functionals on $V$.
\end{abstract}

\section{Introduction}
\label{sec:Introduction}
A $^*$\=/algebra is a unital associative algebra $\mathcal{A}$ over the field of complex numbers
that is endowed with an antilinear involution $\argument^*$ of $\mathcal{A}$ fulfilling $(ab)^* = b^*a^*$
for all $a,b\in \mathcal{A}$. Note that $^*$\=/algebras will always be assumed to have a
unit which is denoted by $\Unit$.
A $C^*$\=/algebras is a $^*$\=/algebra that is complete
with respect to a norm $\seminorm{}{\argument}$ on $\mathcal{A}$ that fulfils
$\seminorm{}{ab} \le \seminorm{}{a} \seminorm{}{b}$ for all $a,b\in \mathcal{A}$
as well as $\seminorm{}{a^*a} = \seminorm{}{a}^2$ for all $a\in\mathcal{A}$.

The Gelfand--Naimark representation theorems \cite[Thm.~1 and Lemma~1]{gelfand.naimark:ImbeddingofNormedRingsIntoRingOfOperators}
are cornerstones of the theory of $C^*$\=/algebras and -- together with the well-behaved spectral theory -- make $C^*$\=/algebras important tools in mathematical physics.
In their simplest form, these two theorems state that all $C^*$\=/algebras
have a faithful representation as $^*$\=/algebras of bounded operators on a Hilbert
space, and that all commutative $C^*$\=/algebras have a faithful representation
as $^*$\=/algebras of bounded complex-valued functions. This allows to interpret $C^*$\=/algebras
as algebras of observables of physical systems: of quantum systems in general, and
in the commutative case of classical systems. From this point of view, the perplexing
differences between the description of quantum systems by means of operators on a
Hilbert space, and of classical systems by functions on a smooth manifold, are just
artefacts of the choice of two different ways to represent the observable algebras.
Consequently, the problem of quantization, i.e.\ of finding a somehow suitable quantum
system to a given classical one, can be formulated in a mathematical precise way,
e.g.\ as finding deformations of commutative $C^*$\=/algebras to non-commutative ones
in the sense of \cite{rieffel:DeformationQuantizationOfHeisenbergManifolds}.

However, the restriction of the Gelfand--Naimark theorems to $C^*$\=/algebras is unfortunate.
While there obviously exist many interesting examples of $^*$\=/algebras of functions or
operators that are not $C^*$\=/algebras, and for which an abstract description might
be desireable, the main motivation for generalizing the Gelfand--Naimark theorems
might again come from physics. Indeed, it is well-known that some of the most basic $^*$\=/algebras
of observables that a (physics-)student gets to know in a course on quantum mechanics are
far away from being $C^*$\=/algebras: If $\mathcal{A}$ is a complex associative algebra
with unit $\Unit$ and $P,Q\in\mathcal{A}$ fulfil the canonical commutation relation
$[P,Q] \coloneqq PQ-QP = \lambda \Unit$ with $\lambda \in \CC\setminus \{0\}$,
then the $n$-fold commutator of $P$ with $Q^n$
fulfils the identity $[P,[P,\dots[P,Q^n]\dots]] = n! \lambda^n \Unit$ for all $n\in \NN$. Thus
there cannot exist a non-trivial submultiplicative seminorm $\seminorm{}{\argument}$
on $\mathcal{A}$, because submultiplicativity would imply at most exponential growth
with $n$ of $\seminorm[\big]{}{ [P,[P,\dots[P,Q^n]\dots]] }$. This rules out any possibility
to embed $\mathcal{A}$ in a $C^*$\=/algebra, and also in many weaker types of
topological $^*$\=/algebras like pro-$C^*$\=/algebras, for which one can prove rather
direct generalizations of the Gelfand--Naimark theorems. Note that faithful representations
of $^*$\=/algebras of canonical commutation relations are well-known and can be given e.g.\ by differential operators. The
problem thus is not to find faithful representations, but to find a sufficiently large
class of $^*$\=/algebras for which the existence of such faithful representations
can be proven by general arguments.

This note gives a solution by focusing not so much on topological properties, but on
order properties of $^*$\=/algebras. This is motivated by \cite{cimpric:representationTheoremForArchimedeanQuadraticModules},
where it was shown that a suitable order on the Hermitian elements of a $^*$\=/algebra $\mathcal{A}$
allows to construct a $C^*$\=/seminorm on the ``bounded'' elements of $\mathcal{A}$, and by \cite{schoetz:PreprintUniversalContinuousCalculusForSuStarAlgebras},
where the continuous calculus and the spectral theory of $C^*$\=/algebras has been extended to certain ordered $^*$\=/algebras.
In the next Section~\ref{sec:pre}, some basic
definitions and results, mainly from locally convex analysis, are recapitulated.
The general idea then is to roughly follow the classical approach from
\cite{kadison:RepresentationTheoremForCommutativeTopologicalAlgebras}:
Section~\ref{sec:order} develops the main result for ordered vector spaces,
namely Theorem~\ref{theorem:orderedvs} which guarantees that on $\sigma$\=/bounded Archimedean ordered vector spaces
there exist many (extremal) positive linear functionals, and which considerably generalizes the result from
\cite{kadison:RepresentationTheoremForCommutativeTopologicalAlgebras}.
Here, an ordered vector space is called ``$\sigma$\=/bounded'' essentially if it contains an increasing sequence of positive elements that
eventually becomes greater than every fixed element. 
Theorem~\ref{theorem:asoperators} shows that every $\sigma$\=/bounded closed ordered $^*$\=/algebra can be represented faithfully
as operators on a (pre-)Hilbert space, and Theorem~\ref{theorem:asfunctions}
shows that in the commutative case, such algebras also admit a faithful representation
as functions if an additional regularity condition (which is clearly necessary) is fulfilled.

Acknowledgements: I would like to thank Profs.\ K.\ Schmüdgen and G.\ Buskes for some helpful discussions and comments on earlier versions of the manuscript.
\section{Preliminaries} \label{sec:pre}
The natural numbers are $\NN \coloneqq \{1,2,3,\dots\}$ and $\NN_0 \coloneqq \NN \cup \{0\}$,
the fields of real and complex numbers are denoted by $\RR$ and $\CC$, respectively.
If $X$ and $Y$ are partially ordered sets (i.e.\ sets together with a reflexive,
transitive and antisymmetric relation $\le$), then a map $\Phi \colon X \to Y$
is called \emph{increasing} if $\Phi(x) \le \Phi(x')$ holds for all $x,x'\in X$
with $x\le x'$. It is called an \emph{order embedding} if it is increasing, injective,
and if additionally $x\le x'$ holds for all $x,x'\in X$ for which $\Phi(x) \le \Phi(x')$.
A partially ordered set $X$ is called \emph{directed} if for all
$x,x' \in X$ there exists a $y\in X$ such that $x\le y$ and $x'\le y$.
Similarly, a subset $S$ of a partially ordered set $X$ is called directed if it is
directed with respect to the order inherited from $X$.
For two vector spaces $V$ and $W$ over the same field of scalars $\FF$, the vector space
of all linear maps from $V$ to $W$ is denoted by $\Lin(V,W)$. In the special case
that $W = \FF$ we write $V^* \coloneqq \Lin(V,\FF)$, the elements of $V^*$ are called
\emph{linear functionals} on $V$, and the evaluation of a linear functional $\omega \in V^*$ on
a vector $v\in V$ is denoted by means of the bilinear \emph{dual pairing}
$\dupr{\argument}{\argument} \colon V^*\times V \to \FF$, $(\omega,v) \mapsto \dupr{\omega}{v}$.

The main technical tools needed in the following are some basic theorems from locally
convex analysis. A \emph{filter} on a set $X$ is a non-empty set $\mathcal{F}$ of
subsets of $X$ with the following two properties:
\begin{itemize}
  \item If $S,T \in \mathcal{F}$, then $S\cap T \in \mathcal{F}$.
  \item If $S \in \mathcal{F}$ and if a subset $T$ of $X$ fulfils $T\supseteq S$, then $T\in\mathcal{F}$.
\end{itemize}
Similarly, a \emph{basis of a filter} on $X$ is a non-empty set $\mathcal{E}$ of subsets
of $X$ such that for all $S,T\in\mathcal{E}$ there exists an $R\in\mathcal{E}$
with $R\subseteq S\cap T$. In this case,
\begin{align}
  \label{eq:fltrgen}
  \filterHull{ \mathcal{E} } \coloneqq \set[\big]{T \subseteq X}{\exists_{S\in \mathcal{E}}: S\subseteq T}
\end{align}
is a filter on $X$, called \emph{the filter generated by $\mathcal{E}$}.
A (real) \emph{topological vector space} is a real vector space $V$ endowed with a
(not necessarily Hausdorff)
topology under which addition $V\times V \to V$ as well as scalar multiplication
$\RR \times V \to V$ are continuous. Then it follows from the continuity of addition
that a subset $U$ of $V$ is a neighbourhood of a vector $v\in V$ if and only if
$U-v \coloneqq \set[\big]{u-v}{u\in U}$ is a neighbourhood of $0$, so the topology
of $V$ is completely described by the $0$-neighbourhoods. The set $\mathcal{N}_0$
of all $0$-neighbourhoods of $V$ is a filter on $V$ and $\mathcal{N}_{0,\mathrm{conv}}$,
the set of all convex $0$-neighbourhoods of $V$, is a basis of a filter on $V$.
In general, $\filterHull{\mathcal{N}_{0,\mathrm{conv}}} \subseteq \mathcal{N}_0$
holds, and $V$ is called a (real) \emph{locally convex vector space} if even
$\filterHull{\mathcal{N}_{0,\mathrm{conv}}} = \mathcal{N}_0$, i.e.\ if the
filter of $0$-neighbourhoods of $V$ is generated by the convex $0$-neighbourhoods.
A \emph{locally convex topology} on $V$ is a topology with which $V$ becomes a locally convex vector space.

There is an easy procedure to construct locally convex topologies on a real vector space $V$:
A subset $S$ of $V$ is called \emph{absorbing} if for every
$v\in V$ there is a $\lambda \in {]0,\infty[}$ such that $\lambda v \in S$, and
it is called \emph{balanced} if $\lambda s \in S$ for all $s\in S$ and all $\lambda \in {[-1,1]}$.
If $\mathcal{N}_{0,\mathrm{abc}}$ is a basis of a filter on $V$ consisting of only
absorbing balanced convex subsets of $V$, then its generated filter
$\filterHull{\mathcal{N}_{0,\mathrm{abc}}}$ is the filter of $0$-neighbourhoods
of a locally convex topology on $V$.
Note that every absorbing balanced convex subset $S$ of $V$ yields a seminorm
$\seminorm{S}{\argument}$ on $V$ by setting
$\seminorm{S}{v} \coloneqq \inf \set[\big]{\lambda \in {]0,\infty[}}{\lambda^{-1} v \in S}$
for all $v\in V$, which is the \emph{Minkowski functional} of $S$.
Conversely, for every seminorm $\seminorm{}{\argument}$ on $V$,
the unit ball $B_{\seminorm{}{\argument}} \coloneqq \set[\big]{v\in V}{\seminorm{}{v} \le 1}$
is an absorbing balanced convex subset of $V$ whose Minkowski functional is again
the original seminorm $\seminorm{}{\argument}$. However, the map from absorbing
balanced convex subsets of $V$ to seminorms on $V$ is only surjective, but not
injective in general. Because of this, there is more freedom in describing locally
convex vector spaces via a basis of the filter of $0$-neighbourhoods than via the
corresponding seminorms.

An important example of Hausdorff locally convex vector spaces is the dual space $V^*$ of an arbitrary
real vector space $V$ endowed with the \emph{weak\=/$^*$-topology}, i.e.\ with
the weakest topology on $V^*$ under which all the evaluation maps
$V^* \ni \omega \mapsto \dupr{\omega}{v} \in \RR$ with $v\in V$ are continuous.
This is the locally convex topology whose filter of $0$-neighbourhoods is generated
by the intersections of unit balls of finitely many seminorms of the form
$V^* \ni \omega \mapsto \abs[\big]{\dupr{\omega}{v}} \in {[0,\infty[}$ with $v\in V$.

The following classical theorems will be crucial for the various representation
theorems. As usual for such results, their proofs typically make use of
the axiom of choice:
\begin{theorem}[Hahn-Banach]
  Let $V$ be a (real) locally convex vector space, $C$ a closed and convex subset of $V$
  and $v \in V\setminus C$. Then there exists a continuous linear functional
  $\omega$ on $V$ such that the inequality $\dupr{\omega}{c} \ge 1+\dupr{\omega}{v}$ holds
  for all $c\in C$.
\end{theorem}
\begin{theorem}[Banach-Alaoglu]
  Let $V$ be a real vector space and $U$ an absorbing subset of $V$, then
  \begin{align}
    U^\circ \coloneqq \set[\big]{\omega \in V^*}{\forall_{u\in U} : \abs{\dupr{\omega}{u}} \le 1}
    \label{eq:BanachAlaoglu}
  \end{align}
  is a convex subset of $V^*$ and compact in the weak-$^*$-topology.
\end{theorem}
If $C$ is a convex subset of a real vector space $V$, then an \emph{extreme point}
of $C$ is an element $e\in C$ with the following property: Whenever $e = \lambda c_1 + (1-\lambda) c_2$
holds for some $c_1,c_2\in C$ and $\lambda \in {]0,1[}$, then $e=c_1 = c_2$. The
set of all extreme points of $C$ will be denoted by $\ex(C)$.
Moreover, if $V$ is a locally convex vector space and $S$ a non-empty subset of $V$, then
$\clconvHull{S}$ will denote the \emph{closed convex hull} of $S$, i.e.\ the closure of
the convex hull of $S$, or equivalently, the intersection of all closed convex subsets of $V$
which contain $S$. In the special case that $V$ is a real vector space and $S$
a non-empty subset of $V^*$, then $\clconvHull{S}$ is understood to be the
closed convex hull of $S$ with respect to the weak-$^*$-topology and
\begin{align}
  \clconvHull{S}
  =
  \set[\Big]{
    \omega \in V^*
  }{
    \forall_{v\in V}: \dupr{\omega}{v} \in \clconvHull{\set[\big]{\dupr{\rho}{v}}{\rho \in S}}
  }
  \label{eq:bipolar}
\end{align}
holds.
This can be seen either by elementary linear algebra or as a consequence of the Hahn-Banach
Theorem and the fact that all weak-$^*$-continuous linear functionals on $V^*$ can be
expressed as maps $V^* \ni \omega \mapsto \dupr{\omega}{v} \in \RR$ with a suitable vector $v\in V$.
\begin{theorem}[Krein-Milman]
  Let $V$ be a real vector space and $K$ a weak\=/$^*$-compact and convex
  subset of $V^*$, then $K = \clconvHull{\ex(K)}$.
\end{theorem}

The tools needed to prove the generalized Gelfand--Naimark theorems are essentially
results about the existence of many (extremal) positive linear functionals
on ordered vector spaces:

An \emph{ordered vector space} is a real vector space $V$ endowed with a partial
order $\le$ such that the two conditions
\begin{align*}
  u+v\le u+w
  \quad\quad\text{and}\quad\quad
  \lambda v\le \lambda w
\end{align*}
hold for all $u,v,w\in V$ with $v\le w$ and all $\lambda \in {[0,\infty[}$.
In this case, $V^+ \coloneqq \set{v\in V}{0\le v}$ is the set of \emph{positive elements}
of $V$, which uniquely determines the order on $V$ because $v \le w$ is equivalent to $w-v\in V^+$
for all $v,w\in V$.
An ordered vector space $V$ is called \emph{Archimedean} if it has the
following property:
Whenever $v\le \epsilon w$ holds for two vectors $v\in V$, $w\in V^+$
and all $\epsilon \in {]0,\infty[}$, then $v\le 0$.
Note also that an ordered vector space $V$ is directed if and only if every $v\in V$
can be decomposed as $v = v_{(+)} - v_{(-)}$ with $v_{(+)}, v_{(-)} \in V^+$. Of course,
such a decomposition is not uniquely determined in general.

If $V$ and $W$ are both ordered vector spaces, then a linear map
$\Phi \colon V\to W$ is increasing if and only if $\Phi(v) \in W^+$ for all $v\in V^+$
and we write
$\Lin(V,W)^+ \coloneqq \set[\big]{\Phi\in\Lin(V,W)}{\Phi\textup{ is increasing}}$.
If $V$ is directed, then there actually exists
a (unique) partial order on $\Lin(V,W)$ such that $\Lin(V,W)$ becomes an ordered
vector space whose positive elements are precisely the increasing linear functions.
However, for simplicity, a linear map $\Phi \colon V \to W$ will always be called \emph{positive} if it
is increasing, even in the case that $V$ is not necessarily directed.
Note that a positive linear map $\Phi \colon V\to W$ is an order embedding
if and only if $\Phi(v) \in W\setminus W^+$ for all $v\in V \setminus V^+$.
The case that $W=\RR$ will be especially interesting:
Then the set of positive linear functionals is denoted by
$V^{*,+} \coloneqq \Lin(V,\RR)^+ = \set[\big]{\omega\in V^*}{\forall_{v\in V^+}:\dupr{\omega}{v}\ge 0}$.
If $V$ is directed, then $V^*$ is ordered vector space itself and a positive linear
functional $\omega$ on $V$ is said to be \emph{extremal}
if for every $\rho\in V^{*,+}$ with $\rho \le \omega$ there is a $\mu \in {[0,1]}$
such that $\rho = \mu \omega$. The set of all extremal positive
linear functionals on $V$ will then be denoted by $V^{*,+,\ex}$ and one can check that
$0\in V^{*,+,\ex}$ and $\lambda \omega \in V^{*,+,\ex}$ for all $\omega \in V^{*,+,\ex}$
and all $\lambda \in {]0,\infty[}$. Note that this definition of extremal positive linear
functionals only makes sense on a directed ordered vector space $V$ because it
refers to the partial order on $V^*$.
There is an extension theorem for (extremal) positive linear functionals from a
sufficiently large linear subspace of a directed ordered vector space to the whole space,
see \cite[Lemma~1.3.2]{schmuedgen:UnboundedOperatorAlgebraAndRepresentationTheory}
for details:
\begin{theorem}[Extension Theorem]
  Let $V$ be a directed ordered vector space and $S$ a linear subspace of $V$ with
  the property that for every $v\in V$ there exists an $s\in S$ such that $0\le s$ and
  $v\le s$. Then $S$ with the order inherited from $V$ is a directed
  ordered vector space and for every $\tilde{\omega} \in S^{*,+}$ there exists an
  $\omega \in V^{*,+}$ that extends $\tilde{\omega}$, i.e.\ that fulfils
  $\dupr{\omega}{s} = \dupr{\tilde{\omega}}{s}$ for all $s\in S$.
  Moreover, in the case that $\tilde{\omega} \in S^{*,+,\ex}$ there even exists
  an $\omega \in V^{*,+,\ex}$ that extends $\tilde{\omega}$.
\end{theorem}
The question of existence of many (extremal) positive linear functionals is
non-trivial in general. More precisely, one asks whether or not the following two
properties are fulfilled:
\begin{definition}
  Let $V$ be an ordered vector space, then we say that the order on $V$ is
  \emph{induced by its positive linear functionals} if for all
  $v \in V\setminus V^+$ there is an $\omega \in V^{*,+}$ such that $\dupr{\omega}{v} < 0$.
\end{definition}
\begin{definition}
 Let $V$ be a directed ordered vector space, then we say that the order on $V$ is
 \emph{induced by its extremal positive linear functionals} if for all
 $v \in V\setminus V^+$ there is an $\omega \in V^{*,+,\ex}$ such that
 $\dupr{\omega}{v} < 0$.
\end{definition}
These two conditions are equivalent to demanding that, for every $v\in V$, the inequality $0 \le v$
holds if and only if $0\le\dupr{\omega}{v}$ for all $\omega \in V^{*,+}$ or for
all $\omega \in V^{*,+, \ex}$, respectively.

We will see shortly that the above conditions of existence of ``many'' (extrmal) positive linear functionals is closely related to the properties of certain locally convex topologies:
Given two elements $\ell$ and $u$ of an ordered vector space $V$, then the \emph{order interval between $\ell$ and $u$} is
\begin{align}
  [\ell, u] \coloneqq \set[\big]{v\in V}{\ell \le v \le u}
  .
\end{align}
Moreover, a subsets $S$ of $V$ is called \emph{saturated} if $[\ell , u] \subseteq S$ is fulfilled for all $\ell, u\in S$ with $\ell \le u$.
For example, every order interval is saturated. It is not hard to see that the intersection of finitely many (even
arbitrarily many) saturated subsets of an ordered vector space $V$ is again
saturated. As a consequence, the set of all absorbing balanced convex and
saturated subsets of $V$ is a basis of the filter of $0$-neighbourhoods of
a locally convex topology:
\begin{definition}
  Let $V$ be an ordered vector space, then the \emph{normal topology}
  on $V$ is the locally convex topology whose filter of $0$-neighbourhoods is
  generated by the absorbing balanced convex and saturated subsets of $V$.
\end{definition}
This topology is not unknown in the theory of ordered vector spaces and $^*$\=/algebras.
For example, the normal topology is essentially what was referred to as $\tau_n$ in
\cite[Sec.~1.5]{schmuedgen:UnboundedOperatorAlgebraAndRepresentationTheory}.

\section{Existence of Extremal Positive Linear Functionals} \label{sec:order}
The next Lemma~\ref{lemma:HahnBanachApplication} is a standard application of the
Hahn-Banach Theorem, a proof is given for convenience of the reader:
\begin{lemma} \label{lemma:HahnBanachApplication}
  Let $V$ be an ordered vector space and assume that there is a locally convex
  topology $\tau$ on $V$ such that $V^+$ is closed with respect to $\tau$. Then
  for every $v\in V \setminus V^+$ there exists a positive linear functional
  $\omega$ on $V$ which fulfils $\dupr{\omega}{v}< 0$ and which is continuous
  with respect to $\tau$.
\end{lemma}
\begin{proof}
  As $V^+$ is convex and closed with respect to $\tau$, the Hahn-Banach Theorem
  implies that for every $v\in V\setminus V^+$ there exists a linear functional $\omega \in V^*$
  which fulfils $\dupr{\omega}{c} \ge 1+ \dupr{\omega}{v}$ for all $c\in V^+$
  and which is continuous with respect to $\tau$. From $0\in V^+$
  it follows that $-1 \ge \dupr{\omega}{v}$. Moreover, $\dupr{\omega}{c} \ge 0$ holds
  for all $c\in V^+$, hence $\omega \in V^{*,+}$:
  Indeed, if there was some $c\in V^+$ with $\dupr{\omega}{c}<0$, then
  $\dupr{\omega}{\lambda c} = \dupr{\omega}{v}$ with $\lambda \coloneqq \dupr{\omega}{v} / \dupr{\omega}{c}$,
  which yields a contradiction because $\lambda \in {]0,\infty[}$ by construction, hence $\lambda c\in V^+$.
\end{proof}
\begin{proposition} \label{proposition:pos}
  Let $V$ be an ordered vector space, then the order on $V$ is induced by its
  positive linear functionals if and only if $V^+$ is closed in $V$ with respect
  to the normal topology.
\end{proposition}
\begin{proof}
  If $V^+$ is closed in $V$ with respect to the normal topology,
  then the order on $V$ is induced by its positive linear functionals as an immediate
  consequence of the previous Lemma~\ref{lemma:HahnBanachApplication}.
  Conversely, assume that the order on $V$ is induced by its positive linear functionals
  and let $v\in V\setminus V^+$ be given. Then there exists an $\omega \in V^{*,+}$
  such that $\dupr{\omega}{v}=-1$ and $U\coloneqq \set[\big]{u\in V}{\dupr{\omega}{u} \in {]{-1},1[}}$
  is an absorbing balanced convex and saturated subset of $V$, so
  $v+U$ is a neighbourhood of $v$ with respect to the normal topology.
  Moreover, application of $\omega$ shows that $(v+U) \cap V^+ = \emptyset$ holds,
  so $V^+$ is closed.
\end{proof}
For finding a sufficient condition under which the order of a directed ordered vector space
is induced by its extremal positive linear functionals, it will be helpful to
be able to decompose a positive linear functional into extremal ones.
We follow essentially the argument of
\cite[Lemmas~12.4.3, 12.4.4]{schmuedgen:UnboundedOperatorAlgebraAndRepresentationTheory}:
\begin{proposition} \label{proposition:extremal}
  Let $V$ be a directed ordered vector space and $U$ an absorbing balanced and
  directed subset of $V$.
  Write $K$ for the set of all those $\omega \in V^{*,+}$ that
  fulfil $\dupr{\omega}{u} \le 1$ for all $u\in U$, then $K$ is weak\=/$^*$-compact and
  $K = \clconvHull{K\cap V^{*,+,\ex}}$.
\end{proposition}
\begin{proof}
  First note that $K = U^\circ \cap V^{*,+}$ with $U^\circ$ like in \eqref{eq:BanachAlaoglu}
  because $U$ is balanced.
  As $U^\circ$ is convex and weak\=/$^*$-compact by the Banach-Alaoglu Theorem, and as $V^{*,+}$ is convex and
  weak\=/$^*$-closed in $V^*$, it follows that $K$ is also convex and weak\=/$^*$-compact.
  The Krein-Milman Theorem then shows that $K = \clconvHull{\ex(K)}$.

  In order to complete the proof it is sufficient to show that $\ex(K) \subseteq V^{*,+,\ex}$.
  Denote the linear span of $K$ in $V^*$ by $W$. Then the map $h \colon W \to {[0,\infty[}$,
  \begin{align*}
    \omega \mapsto h(\omega) \coloneqq \sup \set[\big]{\abs{\dupr{\omega}{u}}}{u\in U}
  \end{align*}
  is a seminorm on $W$ and $K = \set[\big]{\omega \in W\cap V^{*,+}}{h(\omega) \le 1}$.
  Moreover, as $U$ is balanced and directed, even $h(\omega + \omega') \ge h(\omega) + h(\omega')$,
  hence $h(\omega + \omega') = h(\omega) + h(\omega')$,
  holds for all $\omega,\omega' \in W \cap V^{*,+}$.

  Now let $\omega \in \ex(K)$ be given. If $h(\omega) = 0$, then $\omega = 0$ because $U$ is absorbing,
  so $\omega \in V^{*,+,\ex}$ is trivially fulfilled.
  Otherwise $h(\omega) = 1$ because on the one hand, $h(\omega) \le 1$ is clear, and on the other,
  $\omega = h(\omega) \big( h(\omega)^{-1} \omega \big) + \big(1-h(\omega)\big) 0$
  is a representation of $\omega$ as a convex combination of the two elements $h(\omega)^{-1} \omega$
  and $0$ of $K$, which excludes the possibility that $h(\omega) \in {]0,1[}$.
  In this second case that $h(\omega)=1$, consider some $\rho \in V^{*,+}$ that fulfils $\rho \le \omega$.
  Then for all $u\in U$ there exists a $v \in U \cap V^{*,+}$ such that $-v \le u \le v$
  because $U$ is balanced and directed, hence
  $\abs{\dupr{\rho}{u}} \le \dupr{\rho}{v} \le \dupr{\omega}{v}\le 1$. This implies $\rho \in K$,
  and the same estimate with $\omega - \rho$ in place of $\rho$ shows that $\omega-\rho \in K$ as well.
  If $h(\rho) = 0$ or $h(\omega-\rho) = 0$, then $\rho = \mu \omega$ with $\mu = 0$ or $\mu = 1$,
  respectively. Otherwise
  $\omega = h(\rho) \big( h(\rho)^{-1} \rho \big) + h(\omega-\rho) \big( h(\omega-\rho)^{-1} (\omega-\rho) \big)$
  is a representation of $\omega$ as a non-trivial convex combination of the two elements
  $h(\rho)^{-1} \rho$ and $h(\omega-\rho)^{-1} (\omega-\rho)$ of $K$ and thus
  $\rho = \mu \omega$ with $\mu = h(\rho)$. We conclude that $\omega \in V^{*,+,\ex}$ in this case as well.
\end{proof}
\begin{corollary} \label{corollary:extremal}
  Let $V$ be a directed ordered vector space. If there exists a locally convex topology $\tau$ on
  $V$ whose filter of $0$-neighbourhoods has a basis consisting of absorbing balanced convex
  and saturated as well as directed subsets of $V$ and with respect to which $V^+$ is closed,
  then the order on $V$ is induced by its extremal positive linear functionals.
\end{corollary}
\begin{proof}
  Given $v\in V\setminus V^+$, then Lemma~\ref{lemma:HahnBanachApplication} shows that there
  exists a positive linear functional $\omega$ on $V$ such that $\dupr{\omega}{v} < 0$
  and which is continuous with respect to $\tau$. Continuity of $\omega$ implies that
  there exists an absorbing balanced convex and saturated as well as directed subset $U$
  of $V$ such that $\abs[\big]{\dupr{\omega}{u}} \le 1$ for all $u\in U$. Let
  $K\coloneqq \set[\big]{\rho \in V^{*,+}}{\forall_{u\in U} : \dupr{\rho}{u}\le 1}$
  like in the previous Proposition~\ref{proposition:extremal}, then
  $\omega \in K = \clconvHull{K\cap V^{*,+,\ex}}$
  and Equation~\eqref{eq:bipolar} shows that there necessarily exists a $\rho \in K\cap V^{*,+,\ex}$
  which also fulfils $\dupr{\rho}{v} < 0$.
\end{proof}
There is a rather large class of ordered vector spaces for which the normal topology
can not only be described explicitly, but also allows to apply the above
Corollary~\ref{corollary:extremal}:
\begin{definition}
  An ordered vector space $V$ is said to be \emph{$\sigma$\=/bounded} if there exists
  an increasing sequence $(\hat{v}_n)_{n\in \NN}$ in $V^+$ with the property
  that for all $v\in V$ there is an $n\in \NN$ such that $-\hat{v}_n \le v \le \hat{v}_n$
  holds. Such a sequence will be called a \emph{dominating sequence}.
\end{definition}
Note that being $\sigma$\=/bounded can be seen as the combination of two properties: First,
it is required that $V$ is directed, and then, additionally,
that there exists an increasing sequence $(\hat{v}_n)_{n\in \NN}$ in $V^+$ which
has the property that for every $v\in V^+$ there exists an $n\in \NN$
such that $v\le \hat{v}_n$.
\begin{definition} \label{definition:Udelta}
  Let $V$ be a $\sigma$\=/bounded ordered vector space with a dominating sequence
  $(\hat{v}_n)_{n\in \NN}$, and let $(\delta_n)_{n\in \NN}$ be a sequence in ${]0,\infty[}$.
  Then the subset $U_\delta$ of $V$ is defined as the union of increasing order intervals
  \begin{align}
    U_\delta
    \coloneqq
    \bigcup_{N\in \NN} \Big[ -\sum\nolimits_{n=1}^N \delta_n \hat{v}_n , \sum\nolimits_{n=1}^N \delta_n \hat{v}_n \Big]\,.
  \end{align}
\end{definition}
Of course, $U_\delta$ depends not only on the sequence $(\delta_n)_{n\in \NN}$, but
also on the choice of the dominating sequence $(\hat{v}_n)_{n\in \NN}$ and of $V$ itself,
which will always be clear from the context.
\begin{proposition} \label{proposition:Udelta}
  Let $V$ be a $\sigma$\=/bounded ordered vector space with a dominating sequence
  $(\hat{v}_n)_{n\in \NN}$, and let $(\delta_n)_{n\in \NN}$ be a sequence in ${]0,\infty[}$.
  Then $U_\delta$ is an absorbing balanced convex and saturated as well as directed subset of $V$.
\end{proposition}
\begin{proof}
  Given $v\in V$, then there exists an $n\in \NN$ such that $-\hat{v}_n \le v \le \hat{v}_n$
  and therefore $\delta_n v \in U_\delta$. This shows that $U_\delta$ is absorbing.

  Every order interval in $V$ of the form $[-w,w]$ with $w\in V^+$ is easily seen to be
  balanced, convex, saturated and directed. As $U_\delta$ is the union of an
  increasing sequence of such sets, it is balanced, convex, saturated and directed
  itself.
\end{proof}
\begin{proposition} \label{proposition:neighbourhoods}
  Let $V$ be a $\sigma$\=/bounded ordered vector space and $(\hat{v}_n)_{n\in \NN}$
  a dominating sequence in $V$. Then the set of all $U_\delta$ with
  $(\delta_n)_{n\in \NN}$ a sequence in ${]0,\infty[}$ is a basis
  of the filter of $0$-neighbourhoods of the normal topology on $V$.
\end{proposition}
\begin{proof}
  The previous Proposition~\ref{proposition:Udelta} already shows that such a set $U_\delta$
  is a $0$-neighbourhood of the normal topology on $V$ for every
  sequence $(\delta_n)_{n\in \NN}$ in ${]0,\infty[}$.

  Conversely, if $S$ is a $0$-neighbourhood of the normal topology
  on $V$, then there exists a sequence $(\delta_n)_{n\in \NN}$ in ${]0,\infty[}$
  such that $U_\delta \subseteq S$, which can be constructed as follows:
  Let $S'$ be an absorbing balanced convex and saturated subset of $V$ such that $S'\subseteq S$.
  For every $n\in \NN$ there exists an $\epsilon_n \in {]0,\infty[}$ such that
  $\epsilon_n \hat{v}_n \in S'$ because $S'$ is absorbing, and we can define
  $\delta_n \coloneqq 2^{-n} \epsilon_n \in {]0,\infty[}$ for all $n\in \NN$. Then
  $\sum_{n=1}^N \delta_n \hat{v}_n = \sum_{n=1}^N 2^{-n} \epsilon_n \hat{v}_n \in S'$
  for all $N\in \NN$ because $S'$ is balanced and convex. Consequently,
  $\big[-\sum_{n=1}^N \delta_n \hat{v}_n,\sum_{n=1}^N \delta_n \hat{v}_n \big] \subseteq S' \subseteq S$
  for all $N \in \NN$ because $S'$ is also saturated, thus $U_\delta \subseteq S$.
\end{proof}
The above description of the normal topology on a $\sigma$\=/bounded ordered vector space
has essentially already been given in
\cite[Props.~4.1.2, 4.1.3]{schmuedgen:UnboundedOperatorAlgebraAndRepresentationTheory}
and has also been applied indirectly to the decomposition of positive linear functionals
into extremal ones in
\cite[Thm.~12.4.7]{schmuedgen:UnboundedOperatorAlgebraAndRepresentationTheory}.
Even though these results where stated for $O^*$-algebras or $^*$-algebras,
their above generalization to ordered vector spaces does not require any new techniques.
Combining them with a completely order-theoretic characterization of those $\sigma$\=/bounded
ordered vector spaces whose set of positive elements is closed, gives
our first important result:
\begin{theorem} \label{theorem:orderedvs}
  Let $V$ be a $\sigma$\=/bounded ordered vector space, then the following are equivalent:
  \begin{enumerate}
    \item $V$ is Archimedean.   \label{item:orderedvs:Arch}
    \item $V^+$ is closed in $V$ with respect to the normal topology. \label{item:orderedvs:closed}
    \item The order on $V$ is induced by its extremal positive linear functionals. \label{item:orderedvs:ex}
    \item The order on $V$ is induced by its positive linear functionals. \label{item:orderedvs:pos}
  \end{enumerate}
\end{theorem}
\begin{proof}
  Let $(\hat{v}_n)_{n\in \NN}$ be a dominating sequence of $V$, which exists by assumption.

  First consider the case that $V$ is Archimedean. In order to prove the implication
  \refitem{item:orderedvs:Arch} $\Longrightarrow$ \refitem{item:orderedvs:closed},
  we have to show that $V\setminus V^+$ is open with respect to the normal topology
  on $V$. Given $v\in V\setminus V^+$, then
  construct recursively a sequence $(w_n)_{n\in \NN_0}$ in $V\setminus V^+$ and a sequence
  $(\delta_n)_{n\in \NN}$ in ${]0,\infty[}$ as follows: Set $w_0 \coloneqq v$. If $w_{n-1}$
  has been defined for some $n\in \NN$, then choose $\delta_n \in {]0,\infty[}$
  such that $-w_{n-1} \le \delta_n \hat{v}_n$ does not hold, i.e.\ such that
  $w_{n-1} + \delta_n \hat{v}_n \in V\setminus V^+$, and set
  $w_n \coloneqq w_{n-1} + \delta_n \hat{v}_n$.
  Note that such a $\delta_n$ exists because $w_{n-1} \in V\setminus V^+$ and
  because $V$ is Archimedean by assumption.
  From the construction of the sequence $(\delta_n)_{n\in \NN}$ it follows that
  $v+U_\delta \subseteq V\setminus V^+$: Indeed, for every $x \in v+U_\delta$
  there exists an $N\in \NN$ such that $x \le v+\sum_{n=1}^N \delta_n \hat{v}_n = w_N \in V\setminus V^+$,
  hence $x \in V\setminus V^+$. The previous Proposition~\ref{proposition:neighbourhoods}
  now shows that $V\setminus V^+$ is a neighbourhood of $v$ with respect to the
  normal topology on $V$.

  The implication \refitem{item:orderedvs:closed} $\Longrightarrow$ \refitem{item:orderedvs:ex}
  is just an application of Corollary~\ref{corollary:extremal} using that all the subsets
  $U_\delta$ of $V$ with $(\delta_n)_{n\in \NN}$ a sequence in ${]0,\infty[}$ are
  absorbing balanced convex and saturated as well as directed by Proposition~\ref{proposition:Udelta}
  and form a basis of the filter of $0$-neighbourhoods of the normal
  topology on the $\sigma$\=/bounded ordered vector space $V$ by the previous
  Proposition~\ref{proposition:neighbourhoods}.

  Finally, the implication \refitem{item:orderedvs:ex} $\Longrightarrow$ \refitem{item:orderedvs:pos}
  is trivial, and in order to prove \refitem{item:orderedvs:pos} $\Longrightarrow$ \refitem{item:orderedvs:Arch},
  assume that the order on $V$ is induced by its positive linear functionals
  and let $v\in V$ as well as $w\in V^+$ be given such that $v\le \epsilon w$
  for all $\epsilon \in {]0,\infty[}$. Then it follows that $\dupr{\rho}{v} \le 0$ holds
  for all $\rho \in V^{*,+}$ because $\RR$ is Archimedean, thus $v\le 0$.
\end{proof}
One special class of $\sigma$\=/bounded ordered vector spaces are ordered vector spaces $V$
with a \emph{strong order unit} $e$, i.e.\ an element $e\in V^+$ with the property that for all
$v\in V$ there exists $\lambda \in {[0,\infty[}$ such that $v \le \lambda e$.
In this case, $(ne)_{n\in \NN}$ is a dominating sequence, and the existence of positive
linear functionals on Archimedean ordered vector spaces with a strong order unit was already
proven in \cite[Lemma~2.5]{kadison:RepresentationTheoremForCommutativeTopologicalAlgebras}.
The above Theorem~\ref{theorem:orderedvs} generalizes
this classical result to $\sigma$\=/bounded ordered vector spaces and applies
the decomposition of positive linear functionals into extremal ones from
\cite[Thm.~12.4.7]{schmuedgen:UnboundedOperatorAlgebraAndRepresentationTheory}.

It should be unnecessary to point out that there are many examples of important
$\sigma$\=/bounded ordered vector spaces which do not have a strong order unit.
These can be as ordinary as the space of real-valued polynomial functions on $\RR$
with the pointwise order.
\section{Representation of Ordered \texorpdfstring{$^*$-Algebras}{*-Algebras}}
\label{sec:alg}
We now come to the main section which develops the generalized Gelfand--Naimark theorems
for ordered $^*$\=/algebras.
The definition of $^*$\=/algebras has already been given in the introduction. An element
$a$ of a $^*$\=/algebra $\mathcal{A}$ is called \emph{Hermitian} if $a=a^*$ and the
real linear subspace of all Hermitian elements in $\mathcal{A}$ is denoted by $\mathcal{A}_\Hermitian$.
An \emph{ordered $^*$\=/algebra} is a $^*$\=/algebra $\mathcal{A}$ together with
a partial order $\le$ on $\mathcal{A}_\Hermitian$ such that $\mathcal{A}_\Hermitian$
becomes an ordered vector space fulfilling $0 \le \Unit$ and $a^*b\,a \in \mathcal{A}_\Hermitian^+$
for all $a\in \mathcal{A}$, $b\in\mathcal{A}^+_\Hermitian$. This ordered vector
space of Hermitian elements is automatically directed because $4a = (a+\Unit)^2 - (a-\Unit)^2$
and $(a\pm \Unit)^2 \in \mathcal{A}^+_\Hermitian$ hold for all $a\in \mathcal{A}_\Hermitian$.
A \emph{unital $^*$\=/subalgebra} $\mathcal{B}$ of an ordered $^*$\=/algebra $\mathcal{A}$,
i.e.\ a linear subspace $\mathcal{B}$ of $\mathcal{A}$ fulfilling $\Unit \in \mathcal{B}$ as well as
$b^* \in \mathcal{B}$ and $bb' \in \mathcal{B}$ for all $b,b'\in \mathcal{B}$,
is again an ordered $^*$\=/algebra with the order on $\mathcal{B}_\Hermitian$ inherited
from $\mathcal{A}_\Hermitian$.
Ordered $^*$\=/algebras have already been used for understanding representations
of $^*$\=/algebras, e.g.\ in \cite{powers:SelfadjointAlgebrasOfUnboundedOperators2}
or \cite{schmuedgen:UnboundedOperatorAlgebraAndRepresentationTheory}. The set of positive
Hermitian elements of an ordered $^*$\=/algebra is an ``m-admissible cone''
in the language of \cite{schmuedgen:UnboundedOperatorAlgebraAndRepresentationTheory},
or a ``quadratic module'' in the language of (non-commutative) real algebraic geometry.

An ordered $^*$\=/algebra $\mathcal{A}$ is called \emph{$\sigma$\=/bounded} if the ordered
vector space $\mathcal{A}_\Hermitian$ is $\sigma$\=/bounded. Similarly, an 
ordered $^*$\=/algebra $\mathcal{A}$ is said to be \emph{closed} if 
$\mathcal{A}_\Hermitian$ is an Archimedean ordered vector space.
Note that this is not a topological property but rather an order property;
only for $\sigma$-bounded ordered $^*$\=/algebras it follows from Theorem~\ref{theorem:orderedvs}
that $\mathcal{A}$ is closed as an ordered $^*$\=/algebra if and only if the quadratic module $\mathcal{A}^+_\Hermitian$ is closed in $\mathcal{A}_\Hermitian$
with respect to a certain locally convex topology. 
Note also that this property is not at all related to the notion of Archimedean quadratic modules in real algebraic geometry.
It is unfortunate that the term ``Archimedean'' is in several different ways.
Because of this, the term ``Archimedean'' should be avoided in the context of ordered $^*$\=/algebras.
There are some important examples of ordered $^*$\=/algebras that will be relevant in the following:
\begin{example}
  For every set $X$, the space $\CC^X$ of all complex-valued functions
  on $X$ with the pointwise operations and the pointwise order on the Hermitian
  (i.e.\ real-valued) functions is a commutative closed ordered $^*$\=/algebra.
\end{example}

\begin{example}
  If $\Dom$ is a pre-Hilbert space, i.e.\ a complex vector space endowed
  with an inner product $\skal{\argument}{\argument}$,
  antilinear in the first and linear in the second argument, then write $\Adbar(\Dom)$
  for the $^*$\=/algebra of all linear endomorphisms $a$ of $\Dom$ that are adjointable
  in the algebraic sense, i.e.\ for which there exists a (necessarily unique) linear
  endomorphism $a^*$ of $\Dom$ such that
  $\skal{\phi}{a(\psi)} = \skal{a^*(\phi)}{\psi}$ holds for all $\phi,\psi \in \Dom$.
  An element $a$ of $\Adbar(\Dom)$ is Hermitian if and only if $\skal{\phi}{a(\phi)} \in \RR$
  for all $\phi \in \Dom$ and $\Adbar(\Dom)_\Hermitian$ will always be endowed with
  the usual partial order of operators, i.e.\ given $a\in \Adbar(\Dom)_\Hermitian$,
  then $a$ is positive if and only if $\skal{\phi}{a(\phi)} \ge 0$ for all $\phi \in \Dom$.
  This way, $\Adbar(\Dom)$ becomes an closed ordered $^*$\=/algebra which is
  not commutative in general. Unital $^*$-subalgebras of $\Adbar(\Dom)$ are called
  \emph{$O^*$-algebras}.
\end{example}

\begin{example}
  A \emph{$\Phi$-algebra} is an Archimedean ordered vector space $\mathcal{R}$
  in which the supremum $r\vee s \coloneqq \sup \{r,s\}$ and the infimum $r \wedge s \coloneqq \inf \{r,s\}$
  of any two elements $r,s \in \mathcal{R}$ exists, and that is endowed with a
  bilinear product that turns $\mathcal{R}$ into an associative, commutative and unital
  real algebra with the properties that $rs \in \mathcal{R}^+$ for all $r,s\in\mathcal{R}^+$
  and $rs=0$ for all $r,s \in \mathcal{R}$ with $r\wedge s = 0$.
  Then it follows from the properties of $\vee$ and $\wedge$ that
  $(\abs{r}-r)\wedge(\abs{r}+r) = \abs{r}+((-r)\wedge r) = \abs{r}-\abs{r} = 0$, 
  so $\abs{r}^2 - r^2 = (\abs{r}-r)(\abs{r}+r) = 0$,
  i.e.\ $r^2 = \abs{r}^2 \ge 0$ for all $r\in\mathcal{R}$. The complexification
  $\mathcal{A} \coloneqq \mathcal{R} \otimes \CC$ of such a $\Phi$\=/algebra $\mathcal{R}$
  is a commutative ordered $^*$\=/algebra whose
  real linear subspace of Hermitian elements is $\mathcal{A}_\Hermitian \cong \mathcal{R}$.
  For $\Phi$\=/algebras there exists e.g.~a representation theorem by means of
  functions with values in the extended real line $[-\infty,\infty]$, 
  \cite{henriksen.johnson:structureOfArchimedeanLatticeOrderedAlgebras}. 
  In Theorem~\ref{theorem:asfunctions} we will prove a representation theorem by means of
  $\RR$-valued functions under additional assumptions.
\end{example}

\begin{example}
  If $\mathcal{A}$ is a $C^*$\=/algebra, then its Hermitian elements can be
  endowed with a partial order that turns $\mathcal{A}$ into an closed ordered $^*$\=/algebra
  in which the positive Hermitian elements are precisely those Hermitian ones
  whose spectrum is a subset of ${[0,\infty[}$. This is a well-known, but non-trivial result
  in the theory of $C^*$\=/algebras. Showing e.g.\ that the
  sum of two positive Hermitian elements is again a positive Hermitian element required
  some considerable effort in the original proof of the (non-commutative) representation theorem
  for $C^*$\=/algebras in \cite{gelfand.naimark:ImbeddingofNormedRingsIntoRingOfOperators}.
  Moreover, in a $C^*$\=/algebra $\mathcal{A}$, the unit $\Unit$ is a strong
  order unit of $\mathcal{A}_\Hermitian$, i.e.\ for every $a\in \mathcal{A}_\Hermitian$
  there exists a $\lambda \in {[0,\infty[}$ such that $a \le \lambda \Unit$, see
  also the discussion under Theorem~\ref{theorem:orderedvs}.
\end{example}

\begin{example}
  Let $\mathcal{A}$ be a $^*$\=/algebra and define the set
  \begin{align*}
    \mathcal{A}^{++}_\Hermitian
    \coloneqq
    \set[\Big]{\sum\nolimits_{n=1}^N a_n^*a_n}{N\in\NN; a_1,\dots,a_N\in \mathcal{A}}\
  \end{align*}
  of \emph{algebraically positive elements} in $\mathcal{A}$. If $\mathcal{A}^{++}_\Hermitian$
  does not contain a real linear subspace of $\mathcal{A}_\Hermitian$ besides the trivial one $\{0\}$,
  then $\mathcal{A}$ can be turned into an ordered $^*$\=/algebra such that
  $\mathcal{A}^+_\Hermitian = \mathcal{A}^{++}_\Hermitian$. Otherwise, i.e. if
  there is $a\in\mathcal{A}_\Hermitian \setminus \{0\}$ such that both $a$ and $-a$ are
  algebraically positive, there is no possibility to turn $\mathcal{A}$ into an ordered
  $^*$\=/algebra. So the existence of a suitable order on a $^*$\=/algebra
  (especially the antisymmetry of the order) is a non-trivial condition. We will see that this,
  together with two or three further conditions, allows to prove representation theorems similar
  to, but more general than those known for $C^*$\=/algebras.
\end{example}
If $\mathcal{A}$ is a $^*$\=/algebra, then its complex dual vector space $\mathcal{A}^*$ carries
an antilinear involution $\argument^* \colon \mathcal{A}^* \to \mathcal{A}^*$, $\omega \mapsto \omega^*$,
given by $\dupr{\omega^*}{a} \coloneqq \cc{\dupr{\omega}{a^*}}$ for all $a\in\mathcal{A}$.
We say again that an element $\omega\in \mathcal{A}^*$ is \emph{Hermitian} if $\omega^*=\omega$
and write $\mathcal{A}_\Hermitian^*$ for the set of all Hermitian linear functionals
on $\mathcal{A}$, which is a real linear subspace of $\mathcal{A}^*$. Note that a linear
functional $\omega$ on $\mathcal{A}$ is Hermitian if and only if $\dupr{\omega}{a} \in \RR$
holds for all $a\in \mathcal{A}_\Hermitian$. Thus every $\omega \in \mathcal{A}^*_\Hermitian$
can be restricted to an $\RR$-linear functional on $\mathcal{A}_\Hermitian$, and one
can check that this restriction describes an $\RR$-linear isomorphism between the
vector spaces $\mathcal{A}^*_\Hermitian$ and $(\mathcal{A}_\Hermitian)^*$. An
\emph{(extremal) positive Hermitian linear functional} on an ordered $^*$\=/algebra
$\mathcal{A}$ is then defined as a Hermitian linear functional
on $\mathcal{A}$ whose restriction to a (real) linear functional on the ordered vector space
$\mathcal{A}_\Hermitian$ is an (extremal) positive linear functional. The sets of these
(extremal) positive Hermitian linear functionals are denoted by $\mathcal{A}^{*,+}_\Hermitian$
and $\mathcal{A}^{*,+,\ex}_\Hermitian$, respectively, and we say that the order on $\mathcal{A}$
is \emph{induced by its (extremal) positive Hermitian linear functionals} if the order
on $\mathcal{A}_\Hermitian$ is induced by its (extremal) positive linear functionals.
Note that positivity of a Hermitian linear functional $\omega$ on an ordered $^*$\=/algebra $\mathcal{A}$
is in general a stronger condition than just the requirement that $\dupr{\omega}{a^*a} \ge 0$
for all $a\in \mathcal{A}$, which is used quite often in the literature when linear functionals on
general $^*$\=/algebras are discussed.
However, if $\mathcal{A}$ is an ordered $^*$\=/algebra
in which only the algebraically positive Hermitian linear functionals are positive,
i.e.\ if $\mathcal{A}$ is of the type of the fifth example above, then
$\dupr{\omega}{a^*a} \ge 0$ for all $a\in \mathcal{A}$ is also sufficient for a
Hermitian linear functional $\omega$ on $\mathcal{A}$ to be positive.

Positive Hermitian linear functionals on an ordered $^*$\=/algebra $\mathcal{A}$
have some nice properties: Given $\omega \in \mathcal{A}^{*,+}_\Hermitian$
and $a,b\in \mathcal{A}$, then the \emph{Cauchy-Schwarz inequality}
\begin{align}
  \abs[\big]{\dupr{\omega}{b^*a}}^2 &\le \dupr{\omega}{b^*b} \dupr{\omega}{a^*a}
  \shortintertext{holds, as well as}
  \abs[\big]{\dupr{\omega}{a}}^2 &\le \dupr{\omega}{\Unit} \dupr{\omega}{a^*a}
\end{align}
in the special case that $b = \Unit$. This has an important consequence: If a positive
Hermitian linear functional $\omega$ on $\mathcal{A}$ fulfils $\dupr{\omega}{\Unit} = 0$,
then $\omega = 0$. A \emph{state} on $\mathcal{A}$ is a positive Hermitian
linear functional $\omega$ on $\mathcal{A}$ that fulfils $\dupr{\omega}{\Unit} = 1$.
So every $\tilde{\omega} \in \mathcal{A}^{*,+}_\Hermitian \setminus \{0\}$
is a multiple of a unique state $\omega$ on $\mathcal{A}$,
namely of $\omega = \dupr{\tilde{\omega}}{\Unit}^{-1} \tilde{\omega}$.
The set of all states on $\mathcal{A}$ will be denoted by $\States(\mathcal{A})$
and is clearly a convex (possibly empty) subset of the real vector space $\mathcal{A}^*_\Hermitian$.
Again, note that by this definition, states are positive on whole $\mathcal{A}^+_\Hermitian$,
not just on squares.

A map $\Phi \colon \mathcal{A} \to \mathcal{B}$ between two $^*$\=/algebras is said
to be \emph{multiplicative} if $\Phi(aa') = \Phi(a) \Phi(a')$ holds for all $a,a' \in \mathcal{A}$.
Furthermore, it is called a \emph{unital $^*$\=/homomorphism} if it is linear and multiplicative,
maps the unit of $\mathcal{A}$ to the unit of $\mathcal{B}$ and fulfils
$\Phi(a^*) = \Phi(a)^*$ for all $a\in \mathcal{A}$. This last condition is equivalent to
$\Phi(a) \in \mathcal{B}_\Hermitian$ for all $a\in \mathcal{A}_\Hermitian$.
If $\mathcal{A}$ and $\mathcal{B}$ are even ordered $^*$\=/algebras, then such a unital
$^*$\=/homomorphism $\Phi$ is called \emph{positive} or an \emph{order embedding}
if its restriction to an $\RR$-linear map between the ordered vector spaces
$\mathcal{A}_\Hermitian$ and $\mathcal{B}_\Hermitian$
is positive or an order embedding, respectively.

For ordered $^*$\=/algebras we are going to discuss two different types of representations,
which correspond to the first two examples mentioned above:
\begin{definition}
  Let $\mathcal{A}$ be an ordered $^*$\=/algebra. Then a \emph{representation as operators}
  of $\mathcal{A}$ is a tuple $(\Dom, \pi)$ of a pre-Hilbert space $\Dom$ and a
  positive unital $^*$\=/homomorphism $\pi \colon \mathcal{A} \to \Adbar(\Dom)$.
  Similarly, a
  \emph{representation as functions} of $\mathcal{A}$ is a tuple $(X, \pi)$ of a
  set $X$ and a positive unital $^*$\=/homomorphism $\pi \colon \mathcal{A} \to \CC^X$.
  Moreover, such a representation (as operators or as functions) is called
  \emph{faithful} if $\pi$ is even an order embedding.
\end{definition}
Of course, representations as functions are especially interesting for
commutative ordered $^*$\=/algebras. The existence of faithful
representations of an ordered $^*$\=/algebra
$\mathcal{A}$ is closely linked to the question of whether or not the order
on $\mathcal{A}$ is induced by its (extremal) positive Hermitian linear functionals.
\subsection{Representation as Operators}
The well-known construction of the GNS-representation yields a representation
as operators of an ordered $^*$\=/algebra $\mathcal{A}$ out of any positive Hermitian
linear functional on it:
\begin{proposition} \label{proposition:GNS}
  Let $\mathcal{A}$ be an ordered $^*$\=/algebra and $\omega \in \mathcal{A}^{*,+}_\Hermitian$.
  Then the map $\skal{\argument}{\argument}_\omega \colon \mathcal{A} \times \mathcal{A} \to \CC$,
  \begin{align}
    (a,b) \mapsto \skal{a}{b}_\omega \coloneqq \dupr{\omega}{a^* b}
  \end{align}
  is sesquilinear (antilinear in the first, linear in the second argument) and fulfils
  $\cc{\skal{a}{b}_\omega} = \skal{b}{a}_\omega$ as well as $\skal{a}{a}_\omega \in {[0,\infty[}$
  for all $a,b\in \mathcal{A}$. Write $\seminorm{\omega}{\argument}$ for the corresponding
  seminorm on $\mathcal{A}$, defined as $\seminorm{\omega}{a} \coloneqq \skal{a}{a}_\omega^{1/2}$
  for all $a\in \mathcal{A}$, and $\mathrm{ker}\seminorm{\omega}{\argument} \coloneqq \set[\big]{a\in\mathcal{A}}{\seminorm{\omega}{a}=0}$
  for its kernel. Then $\skal{\argument}{\argument}_\omega$ remains well-defined
  on the quotient vector space $\mathcal{A} / \mathrm{ker}\seminorm{\omega}{\argument}$
  on which it describes an inner product.
  Now write $\Dom_\omega$ for the pre-Hilbert space of $\mathcal{A} / \mathrm{ker}\seminorm{\omega}{\argument}$
  with inner product $\skal{\argument}{\argument}_\omega$, and $[b]_\omega \in \mathcal{A} / \mathrm{ker}\seminorm{\omega}{\argument}$
  for the equivalence class of an element $b\in\mathcal{A}$. Then for every $a\in\mathcal{A}$,
  the map $\pi_\omega(a) \colon \Dom_\omega \to \Dom_\omega$,
  \begin{equation}
    [b]_\omega \mapsto \pi_\omega(a)\big( [b]_\omega \big) \coloneqq [ab]_\omega
  \end{equation}
  is a well-defined linear endomorphism of $\Dom_\omega$, it is even adjointable
  with adjoint $\pi_\omega(a^*)$, and the resulting map $\mathcal{A} \ni a \mapsto \pi_\omega(a) \in \Adbar(\Dom_\omega)$
  is a positive unital $^*$\=/homomorphism. Altogether, $\big(\Dom_\omega,\pi_\omega\big)$
  is a representation as operators of the ordered $^*$\=/algebra $\mathcal{A}$.
\end{proposition}
\begin{proof}
  The only detail which is not completely part of the classical GNS-construction
  for $^*$\=/algebras as described e.g.\ in \cite[Sec.~8.6]{schmuedgen:UnboundedOperatorAlgebraAndRepresentationTheory}
  is the observation that $\pi_\omega$ is not only a unital $^*$\=/homomorphism, but
  also positive, because $\skal[\big]{[b]_\omega}{\pi_\omega(a)\,[b]_\omega} = \dupr{\omega}{b^*a\,b}\ge 0$
  for all $a\in\mathcal{A}^+_\Hermitian$.
\end{proof}
\begin{definition}
  Let $\mathcal{A}$ be an ordered $^*$\=/algebra and $\omega \in \mathcal{A}^{*,+}_\Hermitian$,
  then the representation as operators $\big( \Dom_\omega, \pi_\omega \big)$ from
  the previous Proposition~\ref{proposition:GNS} is called the \emph{GNS representation of $\mathcal{A}$
  with respect to $\omega$}.
\end{definition}
The problem of existence of representations as operators of ordered $^*$\=/algebras can
be treated completely analogous to the case of general $^*$\=/algebras:
\begin{proposition} \label{proposition:asoperators}
  Let $\mathcal{A}$ be an ordered $^*$\=/algebra, then there exists a faithful representation
  as operators of $\mathcal{A}$ if and only if the order on $\mathcal{A}$ is
  induced by its positive Hermitian linear functionals.
\end{proposition}
\begin{proof}
  Assume that there exists a faithful representation as operators $(\Dom,\pi)$ of $\mathcal{A}$.
  Given $a\in \mathcal{A}_\Hermitian \setminus \mathcal{A}_\Hermitian^+$, then there exists
  $\phi \in \Dom$ such that $\skal[\big]{\phi}{\pi(a)(\phi)} < 0$.
  But $\mathcal{A} \ni b \mapsto \skal[\big]{\phi}{\pi(b)(\phi)} \in \CC$ is a positive
  Hermitian linear functional. So we conclude that the order on $\mathcal{A}$ is induced
  by its positive Hermitian linear functionals.

  Conversely, assume that the order on $\mathcal{A}$ is induced by its positive Hermitian linear functionals.
  Using the GNS representations of $\mathcal{A}$,
  define the orthogonal sum of pre-Hilbert spaces
  $\Dom_{\mathrm{tot}} \coloneqq \bigoplus_{\omega \in \mathcal{A}^{*,+}_\Hermitian} \Dom_\omega$
  with inner product denoted by $\skal{\argument}{\argument}_{\mathrm{tot}}$,
  as well as for every element $a\in \mathcal{A}$ the linear endomorphism
  $\pi_{\mathrm{tot}}(a) \coloneqq \bigoplus_{\omega\in\mathcal{A}^{*,+}_\Hermitian} \pi_\omega(a)$
  of $\Dom_{\mathrm{tot}}$, i.e.\ $\sum_{\omega\in\mathcal{A}^{*,+}_\Hermitian} [b]_\omega \mapsto
  \pi_{\mathrm{tot}}(a)\big( \sum_{\omega\in\mathcal{A}^{*,+}_\Hermitian} [b]_\omega\big) \coloneqq
  \sum_{\omega\in\mathcal{A}^{*,+}_\Hermitian} \pi_{\omega}(a)\big([b]_\omega\big)$.
  Then it is easy to check that
  $\pi_{\mathrm{tot}}(a)$ is even adjointable with adjoint $\pi_{\mathrm{tot}}(a^*)$
  and that the resulting map
  $\pi_{\mathrm{tot}} \colon \mathcal{A} \to \Adbar\big( \Dom_{\mathrm{tot}} \big)$, $a\mapsto \pi_{\mathrm{tot}}(a)$
  is a positive unital $^*$\=/homomorphism. Moreover, $\pi_{\mathrm{tot}}$ is even an order embedding:
  Indeed, for every
  $a\in \mathcal{A}_\Hermitian \setminus \mathcal{A}_\Hermitian^+$ there exists an $\omega\in\mathcal{A}^{*,+}_\Hermitian$
  such that $\dupr{\omega}{a} < 0$ and thus
  $\skal[\big]{[\Unit]_\omega}{\pi_{\mathrm{tot}}(a)([\Unit]_\omega)}_{\mathrm{tot}} = \dupr{\omega}{a} < 0$.
  It follows that $\big( \Dom_{\mathrm{tot}}, \pi_{\mathrm{tot}} \big)$ is a faithful representation
  as operators.
\end{proof}
Application of Theorem~\ref{theorem:orderedvs} to the above Proposition~\ref{proposition:asoperators}
immediately yields the following generalization of the (non-commutative) Gelfand--Naimark theorem:
\begin{theorem} \label{theorem:asoperators}
  Let $\mathcal{A}$ be a $\sigma$\=/bounded ordered $^*$\=/algebra, then
  $\mathcal{A}$ has a faithful representation as operators if and only if $\mathcal{A}$
  is closed.
\end{theorem}
The original (non-commutative) Gelfand--Naimark theorem is essentially the
special case of this Theorem~\ref{theorem:asoperators} for ordered $^*$\=/algebras $\mathcal{A}$
in which the multiplicative unit $\Unit$ is also a strong order unit. As discussed
under Theorem~\ref{theorem:orderedvs}, $\mathcal{A}$ is automatically $\sigma$\=/bounded in this case.
Note also that the image of a $\sigma$\=/bounded ordered $^*$\=/algebra $\mathcal{A}$
under a faithful representation as operators is an $O^*$\=/algebra with metrizable graph topology
in the language of \cite{schmuedgen:UnboundedOperatorAlgebraAndRepresentationTheory},
which, conversely, are always $\sigma$\=/bounded.
So the above Theorem~\ref{theorem:asoperators} yields an order-theoretic characterization
of the $O^*$\=/algebras with metrizable graph topology. This goes in a similar
direction as \cite{schmuedgen:orderStructureOfTopStarAlgOfUnboundedOperatorsI},
where a topological characterization of a large class of $O^*$\=/algebras has been given.

Despite being valid only in the $\sigma$-bounded case, Theorem~\ref{theorem:asoperators}
still implies that, heuristically, every closed ordered $^*$\=/algebra ``behaves essentially like an $O^*$\=/algebra''.
In order to make this more precise we need the following Lemma:

\begin{lemma} \label{lemma:countablydominatedautomatic}
  Let $\mathcal{A}$ be a $^*$\=/algebra that is generated by a countable subset of $\mathcal{A}$.
  If $\mathcal{A}_\Hermitian$ is endowed with any partial order $\le$ such that $\mathcal{A}$ becomes an ordered
  $^*$-algebra, then $\mathcal{A}$ with this order is automatically $\sigma$\=/bounded.
\end{lemma}
\begin{proof}
  As $\mathcal{A}$ is generated by a countable subset, then $\mathcal{A}$ has at most countable dimensions,
  i.e.~there exists a sequence $(a_n)_{n\in \NN}$ in $\mathcal{A}_\Hermitian$ such that the $\RR$-linear span of $\set{a_n}{n\in \NN}$
  is whole $\mathcal{A}_\Hermitian$. Moreover, from $(\Unit \pm a_n)^2 \in \mathcal{A}_\Hermitian^+$
  it follows that $\pm 2 a_n \le \Unit + a_n^2$ for all $n\in \NN$. So define the
  increasing sequence $(\hat{v}_n)_{n\in \NN}$ in $\mathcal{A}_\Hermitian^+$ as
  $\hat{v}_n \coloneqq n \sum_{k=1}^n (\Unit + a_n^2)$, then $(\hat{v}_n)_{n\in \NN}$ is a
  dominating sequence because for every $b \in \mathcal{A}_\Hermitian$ there exists an
  $n\in \NN$ such that $b$ can be expressed as $b = \sum_{k=1}^n 2\beta_n a_n$ with
  coefficients $\beta_1,\dots,\beta_n \in [{-n},n]$, hence $b \le \hat{v}_n$.
\end{proof}
\begin{proposition}
  Let $\mathcal{A}$ be a closed ordered $^*$\=/algebra and $(a_n)_{n\in \NN}$ any sequence in $\mathcal{A}$.
  Then the unital $^*$\=/subalgebra of $\mathcal{A}$ that is generated by $\set{a_n}{n \in \NN}$
  has a faithful representation as operators.
\end{proposition}
\begin{proof}
  By the previous Lemma~\ref{lemma:countablydominatedautomatic}, the unital $^*$\=/subalgebra of
  $\mathcal{A}$ that is generated by $\set{a_n}{n \in \NN}$ is $\sigma$-bounded,
  so Theorem~\ref{theorem:asoperators} applies.
\end{proof}

\subsection{Representation as Functions}
A slight modification of the well-known Gelfand transformation yields a representation as functions of
any ordered $^*$\=/algebra:
\begin{definition}
  Let $\mathcal{A}$ be an ordered $^*$\=/algebra, then the set of all multiplicative states
  on $\mathcal{A}$, i.e.\ of all positive unital $^*$\=/homomorphisms from $\mathcal{A}$ to $\CC$,
  will be denoted by $\MultStates(\mathcal{A})$.
\end{definition}
\begin{proposition} \label{proposition:Gelfand}
  Let $\mathcal{A}$ be an ordered $^*$\=/algebra. Then
  the map $\pi_{\mathrm{Gelfand}} \colon \mathcal{A} \to \CC^{\MultStates(\mathcal{A})}$, $a \mapsto \pi_{\mathrm{Gelfand}}(a)$
  with $\pi_{\mathrm{Gelfand}}(a) \colon \MultStates(\mathcal{A}) \to \CC$,
  \begin{align}
    \omega \mapsto \pi_{\mathrm{Gelfand}}(a)(\omega) \coloneqq \dupr{\omega}{a}
  \end{align}
  is a positive unital $^*$\=/homomorphism and $\big( \MultStates(\mathcal{A}), \pi_{\mathrm{Gelfand}} \big)$
  is a representation as functions of $\mathcal{A}$.
\end{proposition}
\begin{proof}
  This is an immediate consequence of the properties of the elements in $\MultStates(\mathcal{A})$.
\end{proof}
In order to guarantee the existence of many multiplicative states,
we have to examine states which are at the same time extremal positive Hermitian linear functionals:
\begin{definition}
  Let $\mathcal{A}$ be an ordered $^*$\=/algebra. Then a state $\omega$ on $\mathcal{A}$
  is called \emph{pure} if $\omega$ is also an extremal positive Hermitian linear functional
  on $\mathcal{A}$. The set of all pure states on $\mathcal{A}$ will be denoted by
  $\PureStates(\mathcal{A}) \coloneqq \States(\mathcal{A}) \cap \mathcal{A}^{*,+,\ex}_\Hermitian$.
\end{definition}
The above definition of pure states is equivalent to the more common one as
extreme points of the convex set of states:
\begin{proposition} \label{proposition:pure}
  Let $\mathcal{A}$ be an ordered $^*$\=/algebra, then $\PureStates(\mathcal{A}) = \ex\big( \States(\mathcal{A})\big)$.
\end{proposition}
\begin{proof}
  If $\omega$ is an extreme point of $\States(\mathcal{A})$, then it is also an extremal
  positive Hermitian linear functional, hence a pure state:
  Indeed, given $\rho \in \mathcal{A}^{*,+}_\Hermitian$
  such that $\rho \le \omega$, then, as a consequence of the Cauchy-Schwarz inequality,
  either both $\dupr{\rho}{\Unit}$ and $\dupr{\omega-\rho}{\Unit}$ are in ${]0,1[}$,
  or $\rho = \mu \omega$ with $\mu \in \{0,1\}$. In the former case,
  \begin{align*}
    \omega = \dupr{\omega-\rho}{\Unit} \big( \dupr{\omega-\rho}{\Unit}^{-1} (\omega-\rho) \big) + \dupr{\rho}{\Unit} \big( \dupr{\rho}{\Unit}^{-1} \rho\big)
  \end{align*}
  is a representation of $\omega$ as a non-trivial convex combination
  of the two elements $\dupr{\omega-\rho}{\Unit}^{-1} (\omega-\rho)$ and $\dupr{\rho}{\Unit}^{-1} \rho$
  of $\States(\mathcal{A})$, which implies $\rho = \mu \omega$ with $\mu = \dupr{\rho}{\Unit}$.

  Conversely, if $\omega$ is pure state on $\mathcal{A}$, then it is an extreme point
  of $\States(\mathcal{A})$:
  If $\omega = \lambda \rho + (1-\lambda)\rho'$ with $\rho,\rho' \in \States(\mathcal{A})$
  and $\lambda \in {]0,1[}$, then $\lambda \rho \le \omega$ and $(1-\lambda) \rho' \le \omega$.
  Consequently, there are $\mu,\mu'\in {[0,1]}$ such that
  $\lambda \rho = \mu \omega$ and $(1-\lambda) \rho' = \mu'\omega$.
  Evaluation on $\Unit$ shows that $\lambda = \mu$ and $(1-\lambda) = \mu'$,
  hence $\rho = \omega = \rho'$.
\end{proof}
The sets of pure states and of multiplicative states on an ordered $^*$\=/algebra are closely related.
In order to see this, the following concept will be helpful:
\begin{definition}
  Let $\mathcal{A}$ be an ordered $^*$\=/algebra, $\omega$ a state on $\mathcal{A}$ and $a\in \mathcal{A}$.
  The \emph{variance of $\omega$ on $a$} is defined as
  \begin{align}
    \Var_\omega(a)
    :=
    \dupr[\big]{\omega}{\big(a-\dupr{\omega}{a}\Unit\big)^*\big(a-\dupr{\omega}{a}\Unit\big)}\,.
  \end{align}
\end{definition}
Note that $\Var_\omega(a) \in {[0,\infty[}$ and that the alternative formula $\Var_\omega(a) = \dupr{\omega}{a^*a} - \abs{\dupr{\omega}{a}}^2$ holds for
every state $\omega$ on every ordered $^*$\=/algebra $\mathcal{A}$ and all $a\in\mathcal{A}$.
\begin{proposition} \label{proposition:charactersAreExtreme}
  If $\mathcal{A}$ is an ordered $^*$\=/algebra and $\omega$ a multiplicative state
  on $\mathcal{A}$, then $\omega$ is a pure state on $\mathcal{A}$.
\end{proposition}
\begin{proof}
  By Proposition~\ref{proposition:pure}, the pure states are precisely the extreme points
  of the set of all states. So assume that
  $\rho,\rho' \in\States(\mathcal{A})$ and $\lambda \in {]0,1[}$ fulfil
  $\omega = \lambda \rho + (1-\lambda)\rho'$, then one can check that the identity
  \begin{align*}
    \Var_\omega(a)
    =
    \Var_{\lambda\rho + (1-\lambda)\rho'}(a)
    =
    \lambda \Var_{\rho}(a) + (1-\lambda) \Var_{\rho'}(a) + \lambda(1-\lambda)\abs[\big]{\dupr{\rho-\rho'}{a}}^2
  \end{align*}
  holds for all $a\in \mathcal{A}$. Moreover, $\Var_\omega(a) = 0$ because $\omega$ is multiplicative.
  It follows that $\abs{\dupr{\rho-\rho'}{a}}^2 = 0$ for all $a\in\mathcal{A}$
  because $\Var_{\rho}(a)$ and $\Var_{\rho'}(a)$ are non-negative, so $\rho = \rho' = \omega$.
\end{proof}
In order to be able to obtain a converse statement, we need some more assumptions:
\begin{definition} \label{definition:radical}
  An ordered $^*$\=/algebra $\mathcal{A}$ is called \emph{radical} if it has the following property:
  Whenever $a, b \in \mathcal{A}_\Hermitian$ commute and fulfil $\Unit \le a$ and $0 \le ab$, then $0 \le b$.
\end{definition}
There are some important examples of radical commutative ordered $^*$\=/algebras:
\begin{proposition} \label{proposition:functionsradical}
  Let $\mathcal{A}$ be a commutative ordered $^*$\=/algebra. If $\mathcal{A}$ has a
  faithful representation as functions, then $\mathcal{A}$ is radical and closed.
\end{proposition}
\begin{proof}
  Let $(X,\pi)$ be a faithful represent as functions of $\mathcal{A}$. It is easy
  to check that $\CC^X$ is radical and closed, essentially because $\CC$ is radical and closed.
  Using that $\pi \colon \mathcal{A} \to \CC^X$ is a positive unital $^*$\=/homomorphism
  and an order embedding, it follows immediately that $\mathcal{A}$ has to be radical and closed
  as well.
\end{proof}
\begin{proposition} \label{proposition:phiIsradical}
  Let $\mathcal{A} \coloneqq \mathcal{R}\otimes \CC$ be the complexification
  of a $\Phi$\=/algebra $\mathcal{R}$, then $\mathcal{A}$ is radical.
\end{proposition}
\begin{proof}
  Consider $a, b \in \mathcal{A}_\Hermitian$ such that $\Unit \le a$ and $0 \le ab$.
  We can express $b$ as the difference $b = b_+ -b_-$ of its positive and negative component,
  and from $b_+ \wedge b_- = 0$ it follows that $b_+ b_- = 0$. As $\mathcal{A}^+_\Hermitian$
  is closed under products by definition of $\Phi$\=/algebras, multiplication with $b_-$
  yields $0 \le b_-\,ab = - b_-\,a\,b_- \le 0$, so $b_-\,a\,b_- = 0$.
  From $\Unit \le a$ it now follows that $0 \le (b_-)^2 \le b_-\,a\,b_- = 0$, so $(b_-)^2 = 0$.
  As a consequence, $2\epsilon(\epsilon\Unit - b_-) = (\epsilon\Unit-b_-)^2 + \epsilon^2\Unit \ge 0$,
  hence $b_- \le \epsilon\Unit$, holds for all $\epsilon \in {]0,\infty[}$, which implies 
  $b_- \le 0$ because $\mathcal{A}$ is closed by assumption. This finally shows that 
  $b = b_+-b_- \ge 0$.
\end{proof}
It is also worthwhile to mention an important non-example:
\begin{example}
  Consider the $^*$\=/algebra $\CC[x,y]$ of complex polynomials in two variables $x$ and $y$
  with the $^*$-involution given by complex conjugation of coefficients, thus
  $\CC[x,y]_\Hermitian \cong \RR[x,y]$. On $\CC[x,y]_\Hermitian$ choose the partial order
  that turns $\CC[x,y]$ into an ordered $^*$\=/algebra with cone of positive elements
  given by sums of squares, i.e.
  \begin{align}
    \CC[x,y]_\Hermitian^+
    \coloneqq
    \CC[x,y]_\Hermitian^{++}
    =
    \set[\Big]{\sum\nolimits_{n=1}^N p_n^*p_n}{N\in \NN;\,p_1,\dots,p_N \in \CC[x,y]}\,.
  \end{align}
  Note that the product of two elements of $\CC[x,y]_\Hermitian^+$ is again in $\CC[x,y]_\Hermitian^+$.
  It is well-known that there exist polynomials $p \in \CC[x,y]_\Hermitian \setminus \CC[x,y]_\Hermitian^+$
  which are still pointwise positive on $\RR^2$, i.e.\ $p(s,t) \ge 0$ for all $(s,t) \in \RR^2$.
  An explicit example from \cite{berg.christensen.jensen:RemarkOnTheMultidimensionalMomentProblem}
  is $p \coloneqq x^2y^2(x^2+y^2-\Unit)+\Unit = x^4y^2 + x^2y^4-x^2y^2+\Unit$.
  Now consider $q \coloneqq x^2+y^2+\Unit \in \CC[x,y]_\Hermitian^+$, then
  \begin{align*}
    pq = x^6y^2 + x^4y^4 + x^2y^6  + x^2y^2 + x^2 + y^2 + \big(x^2y^2-\Unit\big)^2 \in \CC[x,y]_\Hermitian^+
  \end{align*}
  is a sum of squares, and thus even $pq^n \in \CC[x,y]_\Hermitian^+$ for all $n\in \NN$,
  and especially $qpq \in \CC[x,y]_\Hermitian^+$. As $q^2 \ge \Unit$ we conclude that
  $\CC[x,y]$ (with this choice for the order) is not a radical ordered
  $^*$\=/algebra and especially does not have a
  faithful representation as functions due to Proposition~\ref{proposition:functionsradical}.
  A closer inspection shows that this is indeed because of the existence of
  ill-behaved pure states, see e.g.\ \cite[Cor.~11.6.4]{schmuedgen:UnboundedOperatorAlgebraAndRepresentationTheory}
  and the discussion there for details.
\end{example}
We proceed with examining the relation between pure states and multiplicative states:
\begin{lemma} \label{lemma:varvanishing}
  Let $\mathcal{A}$ be an ordered $^*$\=/algebra, $\omega$ a state on $\mathcal{A}$ and $a\in \mathcal{A}$
  with $\Var_\omega(a) = 0$, then
  \begin{align}
    \dupr{\omega}{b^*a} = \dupr{\omega}{b^*} \dupr{\omega}{a}
    \quad\quad\text{and}\quad\quad
    \dupr{\omega}{a^*b} = \dupr{\omega}{a^*} \dupr{\omega}{b}
  \end{align}
  hold for all $b\in \mathcal{A}$. A state $\omega$ on $\mathcal{A}$ thus is multiplicative if (and only if)
  $\Var_\omega(a) = 0$ for all $a\in \mathcal{A}$.
\end{lemma}
\begin{proof}
  The Cauchy-Schwarz inequality yields
  \begin{align*}
    \abs[\big]{\dupr{\omega}{a^*b}-\dupr{\omega}{a^*}\dupr{\omega}{b}}^2
    =
    \abs[\big]{\dupr[\big]{\omega}{\big(a-\dupr{\omega}{a}\Unit\big)^*\big(b-\dupr{\omega}{b}\Unit\big)}}^2
    \le
    \Var_\omega(a) \Var_\omega(b)
  \end{align*}
  even for all $a,b\in \mathcal{A}$.
\end{proof}
\begin{lemma} \label{lemma:multsuff}
  Let $\mathcal{A}$ be a commutative ordered $^*$\=/algebra and $\omega \in \States(\mathcal{A})$, then
  $\set[\big]{a\in\mathcal{A}}{\Var_\omega(a)=0}$ is a unital $^*$\=/subalgebra of $\mathcal{A}$
  and is the largest (with respect to inclusion) unital $^*$\=/subalgebra of $\mathcal{A}$ on which
  the restriction of $\omega$ is multiplicative. In the special case that $\Var_\omega(\Unit+a^2) = 0$ holds for all $a\in\mathcal{A}_\Hermitian$
  it follows that $\omega$ is multiplicative on whole $\mathcal{A}$.
\end{lemma}
\begin{proof}
  First assume that $\omega$ is an arbitrary state on $\mathcal{A}$. Then it is easy to check that
  $\Var_\omega(\lambda a) = \abs{\lambda}^2 \Var_\omega(a)$ and also (using the commutativity of $\mathcal{A}$)
  $\Var_\omega(a^*) = \Var_\omega(a)$ hold for all $a\in\mathcal{A}$ and all $\lambda \in \CC$.
  Moreover, if $a,b\in \mathcal{A}$ fulfil $\Var_\omega(a) = \Var_\omega(b) = 0$, then one can check
  with the help of the previous Lemma~\ref{lemma:varvanishing} that
  \begin{align*}
    \Var_\omega(a+b)
    &=
    \dupr[\big]{\omega}{a^*a+a^*b +b^*a+b^*b} - \abs[\big]{\dupr{\omega}{a} + \dupr{\omega}{b}}^2
    =
    0
    \shortintertext{and}
    \Var_\omega(ab)
    &=
    \dupr[\big]{\omega}{b^*a^*a\,b} - \abs[\big]{\dupr[\big]{\omega}{ab}}^2
    =
    0\,.
  \end{align*}
  As $\Var_\omega(\Unit) = 0$ is clearly fulfilled as well,
  one sees that $\set[\big]{a\in\mathcal{A}}{\Var_\omega(a)=0}$ is a unital $^*$\=/subalgebra
  of $\mathcal{A}$.
  From Lemma~\ref{lemma:varvanishing} it also follows that the restriction of $\omega$
  to $\set[\big]{a\in\mathcal{A}}{\Var_\omega(a)=0}$ is multiplicative. Conversely,
  if $\mathcal{B}$ is another unital $^*$\=/subalgebra of $\mathcal{A}$ such that
  the restriction of $\omega$ to $\mathcal{B}$ is multiplicative, then it follows that
  $\Var_\omega(b) = \dupr{\omega}{b^*b} - \abs{\dupr{\omega}{b}}^2 = 0$
  for all $b\in \mathcal{B}$ and thus $\mathcal{B} \subseteq \set[\big]{a\in\mathcal{A}}{\Var_\omega(a)=0}$.

  Now assume that $\Var_\omega(\Unit+a^2) = 0$ holds for all $a\in\mathcal{A}_\Hermitian$.
  As $4b = \big(\Unit + (b+\Unit)^2 \big) - \big(\Unit + (b-\Unit)^2 \big)$
  holds for all $b\in \mathcal{A}_\Hermitian$, it follows from
  $\Var_\omega\big( \Unit + (b\pm\Unit)^2 \big) = 0$ that $\Var_\omega(b) = 0$
  for all $b\in \mathcal{A}_\Hermitian$, hence even $\Var_\omega(c) = 0$ for all
  $c\in \mathcal{A}$ because $c = c_r+\I c_i$ with $c_r \coloneqq (c+c^*)/2 \in \mathcal{A}_\Hermitian$
  and $c_i \coloneqq (c-c^*)/(2\I) \in \mathcal{A}_\Hermitian$. Application of
  Lemma~\ref{lemma:varvanishing} thus shows that $\omega$ is multiplicative
  on whole $\mathcal{A}$.
\end{proof}
\begin{theorem} \label{theorem:pureIsmultiplicative}
  Let $\mathcal{A}$ be a radical commutative ordered $^*$\=/algebra, then $\PureStates(\mathcal{A}) = \MultStates(\mathcal{A})$,
  i.e.\ a state on $\mathcal{A}$ is a pure state if and only if it is multiplicative.
\end{theorem}
\begin{proof}
  Let $\mathcal{\omega}$ be a state on $\mathcal{A}$. If $\omega$ is multiplicative, then
  Proposition~\ref{proposition:charactersAreExtreme} shows that $\omega$ is also a pure state.
  Conversely, if $\omega$ is a pure state, then it remains to show that $\Var_\omega(a^2 + \Unit) = 0$
  for all $a\in \mathcal{A}_\Hermitian$, which, by  the previous Lemma~\ref{lemma:multsuff},
  already implies that $\omega$ is multiplicative.

  So let $a\in\mathcal{A}_\Hermitian$ be given and define the subset
  $S_a \coloneqq \set[\big]{(\Unit + a^2)b}{b\in\mathcal{A}_\Hermitian}$
  of $\mathcal{A}_\Hermitian$. It is clear that $S_a$ is a (real) linear subspace
  of $\mathcal{A}_\Hermitian$. If $(\Unit+a^2)b = (\Unit+a^2)b'$ with $b,b' \in \mathcal{A}_\Hermitian$,
  then $b-b'=0$ because
  $0 = (\Unit+a^2)(b-b') = (a+\I\Unit)(b-b')(a+\I\Unit)^*$ and because $\mathcal{A}$ is radical.
  So every element of $S_a$ is of the
  form $(\Unit+a^2) b$ with a unique $b\in\mathcal{A}_\Hermitian$ and
  the map $\tilde{\rho} \colon S_{a} \to \RR$,
  \begin{align*}
    (\Unit + a^2)b \mapsto \dupr[\big]{\tilde{\rho}}{(\Unit + a^2)b} \coloneqq \dupr{\omega}{b}
  \end{align*}
  is well-defined and is clearly $\RR$-linear. Moreover, for every
  $c\in \mathcal{A}_\Hermitian$ there exists a $(\Unit+a^2)b \in S_a$ with $b\in\mathcal{A}_\Hermitian$ such
  that $0 \le (\Unit+a^2)b$ and $c \le (\Unit+a^2)b$, e.g.\ $(\Unit+a^2)b \coloneqq (\Unit+a^2)(\Unit+c)^2/2$.
  Using again that $\mathcal{A}$ is radical one sees that $\tilde{\rho}$
  is even positive with respect to the order on $S_a$ inherited
  from $\mathcal{A}_\Hermitian$, so the extension theorem for positive linear functionals
  applies and shows that there exists a positive linear functional $\rho$ on $\mathcal{A}_\Hermitian$
  fulfilling $\dupr{\rho}{(\Unit+a^2)b} = \dupr{\omega}{b}$ for all $b\in\mathcal{A}_\Hermitian$.
  Using the isomorphism between $(\mathcal{A}_\Hermitian)^*$ and $\mathcal{A}_\Hermitian^*$
  we can also treat $\rho$ as a positive Hermitian linear functional on $\mathcal{A}$.

  As $\dupr{\rho}{b} \le \dupr[\big]{\rho}{(\Unit+a^2)b} = \dupr{\omega}{b}$ holds
  for all $b\in\mathcal{A}_\Hermitian$ it follows that $\rho \le \omega$, hence
  there exists $\mu \in {[0,1]}$ such that $\rho = \mu \omega$ because $\omega$ is
  a pure state by assumption.
  From evaluation on $\Unit+a^2$ and $(\Unit+a^2)^2$ one gets
  \begin{align*}
    \mu\, \dupr[\big]{\omega}{\Unit+a^2} &= \dupr[\big]{\rho}{\Unit+a^2} = \dupr{\omega}{\Unit} = 1
    \shortintertext{and}
    \mu\, \dupr[\big]{\omega}{(\Unit+a^2)^2} &= \dupr[\big]{\rho}{\big(\Unit+a^2\big)^2} = \dupr[\big]{\omega}{\Unit+a^2}\,,
  \end{align*}
  which yields $\mu \neq 0$ and $\dupr[\big]{\omega}{\Unit+a^2} = \mu^{-1}$
  as well as $\dupr[\big]{\omega}{(\Unit+a^2)^2} = \mu^{-2}$, thus $\Var_\omega(\Unit+a^2) = 0$.
\end{proof}
Similar results about the relation between pure and multiplicative states on certain
$^*$\=/algebras have already occured before, e.g.\ \cite[Thm.~2]{bucy.maltese:RepresentationTheoremForPositiveFunctionals}
for Banach $^*$\=/algebras or \cite[Prop.~11.3.9]{schmuedgen:UnboundedOperatorAlgebraAndRepresentationTheory}
for general $^*$\=/algebras endowed with a special choice of a (pre-)order.
\begin{corollary}
  Let $X$ be a set and $\mathcal{A}$ a unital $^*$\=/subalgebra of $\CC^X$, endowed with
  the pointwise order inherited from $\CC^X$, then the pure states on $\mathcal{A}$ are
  precisely the multiplicative ones.
\end{corollary}
Note, however, that there might be more multiplicative states on $\mathcal{A}$ than just evaluation
functionals at points of $X$.
\begin{corollary} \label{corollary:extremalclosed}
  Let $\mathcal{A}$ be a radical commutative ordered $^*$\=/algebra, then
  \begin{align}
    \mathcal{A}^{*,+,\ex}_\Hermitian
    =
    \set[\big]{
      \omega \in \mathcal{A}^{*,+}_\Hermitian
    }{
      \forall_{a\in\mathcal{A}}: \dupr{\omega}{\Unit}\dupr{\omega}{a^*a} = \abs{\dupr{\omega}{a}}^2
    }
    \label{eq:extremalonstarlag}
  \end{align}
  and the set $\mathcal{A}^{*,+,\ex}_\Hermitian$ of extremal positive Hermitian linear
  functionals on $\mathcal{A}$ is weak\=/$^*$-closed in $\mathcal{A}_\Hermitian^*$.
\end{corollary}
\begin{proof}
  Let $\omega$ be a positive Hermitian linear functional on $\mathcal{A}$. Then
  either $\omega = 0$, in which case $\omega \in \mathcal{A}^{*,+,\ex}_\Hermitian$ and
  $\dupr{\omega}{\Unit}\dupr{\omega}{a^*a} = \abs{\dupr{\omega}{a}}^2$
  for all $a\in\mathcal{A}$ hold, or $\dupr{\omega}{\Unit}>0$. In this second case,
  $\tilde{\omega} \coloneqq \dupr{\omega}{\Unit}^{-1} \omega$ is a state on
  $\mathcal{A}$ and the following chain of equivalences holds:
  $\omega$ is even an extremal positive Hermitian linear functional if and only if
  $\tilde{\omega}$ is a pure state, which is equivalent to $\tilde{\omega}$
  being multiplicative by the Theorem~\ref{theorem:pureIsmultiplicative},
  and Lemma~\ref{lemma:varvanishing} shows that this holds if and only if
  $\Var_{\tilde{\omega}}(a) = 0$ for all $a\in \mathcal{A}$. As
  $\dupr{\omega}{\Unit}^2 \Var_{\tilde{\omega}}(a) =
  \dupr{\omega}{\Unit}\dupr{\omega}{a^*a} - \abs{\dupr{\omega}{a}}^2$, identity
  \eqref{eq:extremalonstarlag} is proven.

  Finally, as $\mathcal{A}^{*,+}_\Hermitian$ is weak\=/$^*$-closed in $\mathcal{A}^*_\Hermitian$
  and as
  $\mathcal{A}^*_\Hermitian \ni \omega \mapsto \dupr{\omega}{\Unit}\dupr{\omega}{a^*a} - \abs{\dupr{\omega}{a}}^2 \in \RR$
  is a weak\=/$^*$-continuous function, one sees that the right hand side of
  \eqref{eq:extremalonstarlag} is weak\=/$^*$-closed in $\mathcal{A}^*_\Hermitian$.
\end{proof}
\begin{corollary} \label{corollary:asfunctions}
  Let $\mathcal{A}$ be a radical commutative ordered $^*$\=/algebra, then the following is equivalent:
  \begin{enumerate}
    \item The Gelfand transformation $\big( \MultStates(\mathcal{A}), \pi_{\mathrm{Gelfand}} \big)$
    of $\mathcal{A}$ discussed in Proposition~\ref{proposition:Gelfand}
    is a faithful representation as functions. \label{item:asfunctions:Gelfand}
    \item There exists a faithful representation as functions of $\mathcal{A}$.
    \label{item:asfunctions:rep}
    \item The order on $\mathcal{A}$ is induced by its extremal positive Hermitian linear functionals.
    \label{item:asfunctions:extremal}
  \end{enumerate}
\end{corollary}
\begin{proof}
  The implication \refitem{item:asfunctions:Gelfand} $\Longrightarrow$ \refitem{item:asfunctions:rep}
  is trivial.

  Assume that there exists a faithful representation as functions $(X,\pi)$ of $\mathcal{A}$
  and let $a\in \mathcal{A}_\Hermitian \setminus \mathcal{A}_\Hermitian^+$ be given. Then
  there exists an $x\in X$ fulfilling $\pi(a)(x) < 0$. But the linear functional
  $\mathcal{A} \ni b \mapsto \pi(b)(x) \in \CC$ is a multiplicative state, hence a pure state by
  Proposition~\ref{proposition:charactersAreExtreme}, and thus especially an extremal positive Hermitian linear
  functional on $\mathcal{A}$. This proves the implication \refitem{item:asfunctions:rep} $\Longrightarrow$
  \refitem{item:asfunctions:extremal}.

  Finally, if the order on $\mathcal{A}$ is induced by its extremal positive Hermitian linear functionals,
  then for every $a\in \mathcal{A}_\Hermitian \setminus \mathcal{A}_\Hermitian^+$ there exists an
  extremal positive Hermitian linear functional $\tilde{\omega}$ on $\mathcal{A}$ such that $\dupr{\tilde{\omega}}{a} < 0$.
  From $\tilde{\omega} \neq 0$ it follows that $\dupr{\tilde{\omega}}{\Unit} > 0$ and therefore
  $\omega \coloneqq \dupr{\tilde{\omega}}{\Unit}^{-1} \tilde{\omega}$ is a well-defined pure
  state on $\mathcal{A}$, thus also a multiplicative state by Theorem~\ref{theorem:pureIsmultiplicative}.
  So we see that $\pi_{\mathrm{Gelfand}}(a)$ cannot be positive because $\pi_{\mathrm{Gelfand}}(a)(\omega) < 0$,
  and conclude that the Gelfand transformation of $\mathcal{A}$ is faithful,
  i.e.\ \refitem{item:asfunctions:extremal} $\Longrightarrow$ \refitem{item:asfunctions:Gelfand} holds.
\end{proof}
There are some representation theorems for commutative $^*$\=/algebras endowed with a locally convex topology
defined by submultiplicative seminorms, i.e.\ seminorms fulfilling the estimate
$\seminorm{}{ab} \le \seminorm{}{a}\seminorm{}{b}$ for all elements $a$ and $b$ of the algebra.
One example is of course the commutative Gelfand--Naimark theorem. However, for radical commutative
ordered $^*$\=/algebras, the above Corollary~\ref{corollary:asfunctions} combined with
Corollary~\ref{corollary:extremal} yields an approach using a rather different type of locally convex
topologies.

Application of Theorem~\ref{theorem:orderedvs} to the above Corollary~\ref{corollary:asfunctions} immediately yields the following
generalization of the commutative Gelfand--Naimark theorem:

\begin{theorem} \label{theorem:asfunctions}
  Let $\mathcal{A}$ be a $\sigma$\=/bounded radical commutative ordered $^*$\=/algebra, then
  $\mathcal{A}$ has a faithful representation as functions if and only if $\mathcal{A}$ is closed.
\end{theorem}
Like before, this implies that, heuristically, every closed radical commutative ordered $^*$\=/algebra ``behaves essentially like a $^*$\=/algebra of functions'':
\begin{corollary}
  Let $\mathcal{A}$ be a closed radical commutative ordered $^*$\=/algebra and $(a_n)_{n\in \NN}$ any sequence in $\mathcal{A}$.
  Then the unital $^*$\=/subalgebra of $\mathcal{A}$ that is generated by $\set{a_n}{n \in \NN}$
  has a faithful representation as functions.
\end{corollary}
\begin{proof}
  By Lemma~\ref{lemma:countablydominatedautomatic}, the unital $^*$\=/subalgebra of
  $\mathcal{A}$ that is generated by $\set{a_n}{n \in \NN}$ is $\sigma$-bounded,
  so Theorem~\ref{theorem:asfunctions} applies.
\end{proof}
The two Theorems \ref{theorem:asoperators} and \ref{theorem:asfunctions}
allow to develop a theory of $^*$\=/algebras of (possibly unbounded) operators
in a representation-independent way. While the assumption of being $\sigma$\=/bounded
might be replaced by others, e.g.\ the existence of a well-behaved locally convex
topology, the property of being radical is necessary for a commutative
ordered $^*$\=/algebra to have a faithful representation as functions, see
Proposition~\ref{proposition:functionsradical}.

\end{onehalfspace}
\end{document}